\documentclass[a4paper,reqno]{amsart}
\usepackage[utf8]{inputenc}
\usepackage[T1]{fontenc}
\usepackage{lmodern}
\usepackage{amssymb,amsxtra}
\usepackage{nicefrac,mathtools,enumitem}
\usepackage{mdwlist}
\setlist[enumerate,1]{label=\textup{(\arabic*)}}% ensure that enumerations in theorems are upright

\usepackage{tikz}
\usetikzlibrary{matrix}
\tikzset{cd/.style=matrix of math nodes,row sep=2em,column sep=2em, text height=1.5ex, text depth=0.5ex}
\tikzset{cdar/.style=->,auto}
\tikzset{overar/.style={draw=white,double=black,double distance=.4pt,very thick}}

\usepackage{dsfont}

\usepackage{microtype}
\usepackage[pdftitle={Quantum $E(2)$ groups for complex deformation parameters},
  pdfauthor={Atibur Rahaman and Sutanu Roy},
  pdfsubject={Mathematics; MSC }
]{hyperref}
\usepackage[lite]{amsrefs}

\renewcommand{\PrintDOI}[1]{\href{http://dx.doi.org/\detokenize{#1}}{doi: \detokenize{#1}}}

\BibSpec{book}{%
  +{}  {\PrintPrimary}                {transition}
  +{,} { \textit}                     {title}
  +{.} { }                            {part}
  +{:} { \textit}                     {subtitle}
  +{,} { \PrintEdition}               {edition}
  +{}  { \PrintEditorsB}              {editor}
  +{,} { \PrintTranslatorsC}          {translator}
  +{,} { \PrintContributions}         {contribution}
  +{,} { }                            {series}
  +{,} { \voltext}                    {volume}
  +{,} { }                            {publisher}
  +{,} { }                            {organization}
  +{,} { }                            {address}
  +{,} { \PrintDateB}                 {date}
  +{,} { }                            {status}
  +{}  { \parenthesize}               {language}
  +{}  { \PrintTranslation}           {translation}
  +{;} { \PrintReprint}               {reprint}
  +{.} { }                            {note}
  +{.} {}                             {transition}
  +{} { \PrintDOI}                   {doi}
  +{} { available at \url}            {eprint}
  +{}  {\SentenceSpace \PrintReviews} {review}
}

\numberwithin{equation}{section}

\theoremstyle{plain}
\newtheorem{theorem}[equation]{Theorem}
\newtheorem{lemma}[equation]{Lemma}
\newtheorem{proposition}[equation]{Proposition}

\newtheorem{corollary}[equation]{Corollary}

\theoremstyle{definition}
\newtheorem{definition}[equation]{Definition}

\theoremstyle{remark}
\newtheorem{remark}[equation]{Remark}

\newtheorem{example}[equation]{Example}

\newcommand{\tenscorep}{\mathbin{\begin{tikzpicture}[baseline,x=.75ex,y=.75ex] \draw[line width=.2pt] (-0.8,1.15)--(0.8,1.15);\draw[line width=.2pt](0,-0.25)--(0,1.15); \draw[line width=.2pt] (0,0.75) circle [radius = 1];\end{tikzpicture}}}

\newcommand*{\Braiding}[2]{\begin{tikzpicture}[baseline]
    \draw[-] (0,0) -- (1.4ex,1.4ex) node[right,inner sep=0pt] {$\scriptstyle #2$};
    \draw[-,draw=white,line width=2.4pt] (0,1.4ex) -- (1.4ex,0);
    \draw[-] (1.4ex,0) -- (0,1.4ex) node[left,inner sep=0pt] {$\scriptstyle #1$};
  \end{tikzpicture}}

\newcommand*{\corep}[1]{\textup{#1}}          %Corepresentation of a quantum group
\newcommand*{\Corep}[1]{\mathbb{#1}}          %Corepresentation as operator on Hilbert space
\newcommand*{\DuCorep}[1]{\hat{\Corep{#1}}}   %Dual corepresentation viewed as operator on Hilbert space

\newcommand*{\Rmattxt}{R}%R-matrix for text use

\newcommand*{\Bialg}[1]{(#1,\Comult[#1])}%C*-bialgebra
\newcommand*{\YDCorepcat}{\mathcal{YD}\mathfrak{Rep}}%Category of Yetter-Drinfeld corepresentations

\newcommand*{\dom}{\mathcal D}% domain of an unbounded linear map
\newcommand*{\Sp}{\text{Sp}}%Spectrum

\newcommand*{\CLS}{{-\textup{closed linear span}}}
\newcommand*{\nb}{\nobreakdash}
\newcommand*{\Star}{$^*$\nb-}

\newcommand*{\C}{\mathbb C}
\newcommand*{\Z}{\mathbb Z}
\newcommand*{\R}{\mathbb R}
\newcommand*{\N}{\mathbb N}

\newcommand*{\T}{\mathbb T}

\newcommand*{\G}[1][G]{\mathbb #1}%quantum group
\newcommand*{\DuG}[1][G]{\widehat{\mathbb{#1}}}
\newcommand*{\qSU}{\textup{SU}_{\textup{q}}(2)}%quantum SU
\newcommand*{\qU}{\textup{U}_{\textup{q}}(2)}%quantum U2
\newcommand*{\qE}{\textup{E}_{\textup{q}}(2)}%quantum E2
\newcommand*{\Comult}[1][]{\Delta_{#1}}%comultiplication
\newcommand*{\DuComult}[1][]{\hat{\Delta}_{#1}}%dualcomultiplication
\newcommand*{\Qgrp}[2]{\mathbb{#1}=(#2,\Comult[#2])}%quantum group
\newcommand*{\DuQgrp}[2]{\widehat{\mathbb{#1}}=(\hat{#2},\DuComult[#2])}%dual quantum group
\newcommand*{\Mod}[1]{\abs{#1}}% modulus of polar decomposition
\newcommand*{\Ph}[1]{\Phi_{#1}}% phase of polar decomposition

\newcommand*{\Bound}{\mathbb B}%adjointable operators on a Hilbert module
\newcommand*{\Comp}{\mathbb K}%compact operators on a Hilbert module
\newcommand*{\transpose}{\mathsf T}%transpose

\newcommand*{\Contvin}{\textup C_0}%continuous functions vanishing at infinity
\newcommand*{\Cont}{\textup C}%continuous functions
\newcommand*{\QuantExp}[1][]{F_{|q|}{#1}}%Woronowicz's quantum exponential function

\newcommand*{\Mor}{\textup{Mor}}%nondegenerate *-homomorphisms of C*-algebras
\newcommand*{\Id}{\textup{id}}%identity map

\newcommand*{\Multunit}[1][]{\mathbb W}%muliplicative unitary on action on a hilbert space
\newcommand*{\multunit}[1][]{\textup W^{#1}}%multiplicative unitary as a bicharacter of the multiplier algebra
\newcommand*{\BrMultunit}{\mathbb F}%braided muliplicative unitary on action on a hilbert space
\newcommand*{\Rmat}{\textup R}%R-matrix
\newcommand*{\Flip}{\Sigma}% flip operator on Hilbert space
\newcommand*{\flip}{\sigma}% flip map on the multiplier algebra
\newcommand*{\Cst}{\textup C^*}%C*-algebra
\newcommand*{\Cred}{\textup C^*_\textup r}%reduced group C*-algebra

\newcommand*{\Cstcat}{\mathfrak{C^*alg}}%category of C*-algebras
\newcommand*{\Corepcat}{\mathfrak{Rep}}
\newcommand*{\Hils}[1][H]{\mathcal{#1}}%Hilbert space
\newcommand*{\Mult}{\mathcal M}%multiplier algebra
\newcommand*{\U}{\mathcal U}%unitary group

\newcommand*{\defeq}{\mathrel{\vcentcolon=}}

\newcommand*{\abs}[1]{\lvert#1\rvert}
\newcommand*{\norm}[1]{\lVert#1\rVert}
\newcommand*{\conj}[1]{\overline{#1}}
\newcommand*{\cl}[1]{\overline{#1}}
\newcommand*{\Specn}{\cl{\C}^{|q|}}
\newcommand*{\Aff}{\sp{ }\eta\sp{ }}

\DeclareMathOperator{\Hom}{Hom}% hom

\makeatletter
\@namedef{subjclassname@2020}{%
  \textup{2020} Mathematics Subject Classification}
\makeatother

\begin{document}
\title[Quantum $\textup{E(2)}$ groups for complex deformation parameters]{Quantum $\textup{E(2)}$ groups for complex deformation parameters}

\author{Atibur Rahaman}
\email{atibur.rahaman@niser.ac.in}
\address{School of Mathematical Sciences\\
 National Institute of Science Education and Research  Bhubaneswar, HBNI\\
 Jatni, 752050\\
 India}

\author{Sutanu Roy}
\email{sutanu@niser.ac.in}
\address{School of Mathematical Sciences\\
 National Institute of Science Education and Research  Bhubaneswar, HBNI\\
 Jatni, 752050\\
 India}

\begin{abstract}
 We construct a family of~\(q\) deformations of~\(\textup{E}(2)\) group for nonzero complex parameters~\(\Mod{q}<1\) as locally compact braided quantum groups over the circle group~\(\T\) viewed as a quasitriangular quantum group with respect to the 
 unitary~\(\Rmattxt\)\nb-matrix~\(\Rmat(m,n)\defeq (q/\bar{q})^{mn}\) for all~\(m,n\in\Z\). For real~\(0<\Mod{q}<1\), the deformation  coincides with Woronowicz's~\(\qE\) groups. As an application, we study the braided analogue of the contraction procedure between \(\textup{SU}_{q}(2)\) and~\(\textup{E}_{q}(2)\) groups in the spirit of Woronowicz's quantum analogue of the classic In\"on\"u\nb-Wigner group contraction. 
 Consequently, we obtain the bosonisation of braided \(\textup{E}_{q}(2)\) groups by 
 contracting \(\textup{U}_{q}(2)\) groups.
\end{abstract}

\thanks{The first author was supported by a UGC NET-JRF fellowship, India. The second author was partially 
supported by an Early Career Research Award given by SERB-DST, Government of India Grant No. 
ECR/2017/001354.}

\maketitle

\section{Introduction}
\label{sec:introduction}
Quantum~\(\textup{E(2)}\) groups were constructed as \(q\)\nb-deformations of the double cover of the group of motions of the Euclidean plane \(\textup{E}(2)\) for real deformation parameters~\(0<q<1\)~\cites{W1991a,W1991b}. They were used as a prototype for constructing a large class of examples of locally compact quantum groups~\cites{W2001,WZ2002}. Here we construct a family of \(q\)-deformations of the same group for complex deformation parameters~\(0<\Mod{q}<1\). Consider 
the pair of operators~\((v,n)\) satisfying the following conditions
\begin{equation}
 \label{eq:gen-E2}
 v\text{ is unitary},
 \quad
 n\text{ is normal},
 \quad
  vnv^{*}=qn, 
  \quad
  \Sp(n)=\Specn ,
\end{equation}
where~\(\Specn:=\big\{\lambda\in\C\colon |\lambda|\in |q|^{\Z}\big\}\cup \{0\}\). 
Let \(q=\Mod{q}\Ph{q}\) and \(n=\Mod{n}\Ph{n}\) be the polar decompositions 
of~\(q\) and~\(n\), respectively. The explicit meaning of the commutation relation \(vnv^{*}=qn\) is \(v\Ph{n}v^{*}=\Ph{q}\Ph{n}\) and~\(v\Mod{n}v^{*}=\Mod{q}\Mod{n}\). 

The spectral condition of~\(n\) implies that \(\Contvin(\Specn)\) is generated by 
the unbounded normal operator~\(n\) in the sense of~\cite{W1995}*{Definition 3.1}. 
The commutation relation~\(vnv^{*}=qn\) induces the following action~\(\alpha\colon\Z\to\textup{Aut}(\Contvin(\Specn))\): 
\begin{equation}
 \label{eq:Z-action}
 \alpha_{k}(f)(\lambda)=f(q^k\lambda), 
\quad\text{for all~\(k\in\Z\), \(f\in \Contvin(\Specn)\) and \(\lambda\in\Specn\).}
\end{equation}
Let~\(B=\Contvin(\Specn)\rtimes_{\alpha} \Z\) be the associated crossed product~\(\Cst\)\nb-algebra. 
By virtue of \cite{W1991b}*{Theorem 1.6} and the universal property of the \(\Cst\)\nb-algebra 
crossed products, the pair of operators \((v,n)\) in~\eqref{eq:gen-E2} generate the \(\Cst\)\nb-algebra~\(B\). In particular, \(v,n \notin B\) but these are affiliated with~\(B\), see~\cite{W1991b}*{Definition 1.1}.

For real \(0<q<1\), one has a well\nb-defined nondegenerate \Star{}homomorphism 
\(\Comult[B]\colon B\to \Mult(B\otimes B)\) satisfying
\begin{equation} 
 \label{eq:comult-E2}
  \Comult[B](v)=v\otimes v, 
  \qquad
  \Comult[B](n)=v\otimes n\dotplus n\otimes v^{*}.
\end{equation}
Here~\(\otimes\), \(\Mult(B)\) and \(\dotplus\) denote the minimal tensor product of~\(\Cst\)\nb-algebras, multiplier algebra of~\(B\) and the closure of 
the sum of the operators, respectively. Also, \(\Comult[B](n)\) is affiliated with~\(B\otimes B\). The pair~\(\Bialg{B}\) is 
Woronowicz's \(\qE\) group~\cites{W1991a,W1991b}. 
%It is constructed from a manageable multiplicative unitary as ~\cite{W1996}. 

However, the passage from the real to the complex deformation parameter~\(q\) reveals two interesting features. Firstly, the \(\Cst\)\nb-algebra \(B\) for the deformation 
parameter~\(\Mod{q}\) is isomorphic to the one for the complex parameter~\(q\), see Proposition~\ref{Prop:isom-CstEq2}. Secondly, \(\Comult[B](n)\) 
and~\(\Comult[B](n)^{*}\) do not commute (even formally). In other words, \(\Bialg{B}\) fails to be a \(\Cst\)\nb-quantum group 
unless~\(q=\overline{q}\).

We address this issue by replacing the ordinary tensor product~\(\otimes\) by the twisted tensor product~\(\boxtimes_{\Rmat}\) from~\cite{MRW2016} associated to the bicharacter~\(\Rmat\colon\Z\times\Z\to\T\) defined by
\begin{equation}
 \label{eq:bichar}
 \Rmat(m,n)\defeq \zeta^{mn}, 
 \quad\text{for all~\(m,n\in\Z\) with \(\zeta\defeq\frac{q}{\overline{q}}=(\Ph{q})^{2}\).}
\end{equation}
For that matter, we view the circle group~\(\T\) as a quasitriangular quantum group with respect to~\(\Rmat\) as explained in Example~\ref{ex:T-quasitriag}.
% Let~\(\{e_{p}\}_{p\in\Z}\) be an orthonormal basis of~\(\ell^{2}(\Z)\). Then \(\Cont(\T)\) is the (universal) \(\Cst\)\nb-algebra generated by the right shift operator \(ze_{p}\defeq e_{p+1}\) on~\(\ell^{2}(\Z)\). Also, \(\Cont(\T)\), along with the comultiplication map \(\Comult[\Cont(\T)](z)\defeq z\otimes z\), is a (compact) quantum group and denote it by~\(\G\).

Now consider the category~\(\Cstcat(\T)\) whose objects 
are \(\Cst\)\nb-algebras equipped with an action of~\(\T\) and morphisms are~\(\T\)\nb-equivariant nondegenerate \Star{}homomorphisms. Clearly, \(\Cont(\T)\) is an object 
of~\(\Cstcat(\T)\) with respect to the translation action of~\(\T\). Also, the universal property of~\(B\) ensures that~\((v,\lambda)\to v\) and \((n,\lambda)\to\lambda n\) for all~\(\lambda\in\T\) extends uniquely to an action of~\(\T\) on \(B\) (see also~\eqref{eq:T-act_Eq2}). The monoidal structure~\(\boxtimes_{\Rmat}\) on \(\Cstcat(\T)\) is defined as follows: \(C\boxtimes_{\Rmat}D\) 
is a \(\Cst\)\nb-algebra for all objects \((C,\delta_{C})\), \((D,\delta_{D})\) 
of \(\Cstcat(\T)\) with injective nondegenerate \Star{}homomorphisms 
\(j_{1}\colon C\to\Mult(C\boxtimes_{\Rmat}D)\) and \(j_{2}\colon D\to\Mult(C\boxtimes_{\Rmat}D)\) such that the \Star{}algebra 
generated by~\(\{j_{1}(c)j_{2}(d)\mid c\in C, d\in D\}\) is dense in~\(C\boxtimes_{\Rmat}D\). On the homogeneous elements~\(c_{k}\in C\) and~\(d_{l}\in D\) of degree~\(k\) and~\(l\), that is~\(\delta_{C}(c_{k})=c_{k}\otimes z^{k}\) and~\(\delta_{D}(d_{l})=d_{l}\otimes z^{l}\), the canonical embeddings \(j_{1},j_{2}\) commute up to~\(\zeta^{kl}\): \(j_{1}(c_{k})j_{2}(d_{l})=\zeta^{kl}j_{2}(d_{l})j_{1}(c_{k})\).

In particular, we consider the twofold twisted tensor product 
\(B\boxtimes_{\Rmat}B\) and define an analogue of 
\eqref{eq:comult-E2}  in~\(\Cstcat(\T)\) by
\begin{equation}
 \label{eq:comult-bE2}
 \Comult[B](v)\defeq j_{1}(v)j_{2}(v), 
 \qquad
 \Comult[B](n)\defeq j_{1}(v)j_{2}(n)\dotplus j_{1}(n)j_{2}(v^{*}).
\end{equation}
The primary goal of this article is to prove that~\(\qE=\Bialg{B}\) is a braided \(\Cst\)\nb-quantum group over the quasitriangular quantum group~\(\T\) with respect to~\(\Rmat\). This is essentially contained in Theorem~\ref{the:BrdqE2}. Subsequently, we construct 
an ordinary \(\Cst\)\nb-quantum group~\(\Qgrp{H}{C}\) with an idempotent quantum group homomorphism with image~\(\T\). In the Hopf\nb-algebraic context this process is known  as \emph{bosonisation} which was discovered by Radford~\cite{R1985} and extensively studied by Majid~\cites{M1994, M1999, M2000}. Therefore, \(\Qgrp{H}{C}\) is the analytic counter part of the bosonisation of~\(\qE\), which is a new example of (non\nb-compact) \(\Cst\)\nb-quantum group.

The double cover of \(\textup{E}(2)\) may be constructed by contracting~\(\textup{SU}(2)\). This is one of the well studied examples of the classic In\"on\"u\nb-Wigner group contraction~\cite{IW1953}.  In the purely algebraic setting, the quantum analogue of the contraction procedure was investigated by Celeghini, Giachetti, Sorace and Tarlini~\cites{CGST1990, CGST1991, CGST1991a} and applied to construct new examples of quantum groups. Their work motivated Woronowicz~\cite{W1992} to study the contraction procedure for \(\Cst\)\nb-quantum groups. In this article, it was shown that for real~\(0<q<1\), the contraction of the compact quantum groups \(\qSU\) respect to the (closed quantum) subgroup \(\T\) coincides with~\(\qE\) groups.  

\medskip 

Deformations of~\(\textup{SU}(2)\) group for complex~\(q\) satisfying~\(0<\Mod{q}<1\) was considered in~\cite{KMRW2016}. They are braided compact quantum groups over the quasitriangular quantum group~\(\T\) with respect to~\(\Rmat\) and \(\textup{U}_{q}(2)\) groups are the associated bosonisation. 
 Now \(\T\) is a  \emph{closed quantum subgroup} of braided~\(\qSU\) groups~\cite{S2020}*{Proposition 4.1}  in the sense of Woronowicz~\cite{DKSS2012}*{Definition 3.2}. Similarly, we observe that \(\T\) is also a closed quantum subgroup of braided~\(\qE\) in Remark~\ref{rem:closed-qnt}. Then we prove a braided analogue of the contraction procedure between~\(\qSU\) and~\(\qE\) in Theorem~\ref{the:contraction}. More precisely, the \(\T\)\nb-equivariant contraction of braided~\(\qSU\) groups with respect to~\(\T\) are isomorphic to braided~\(\qE\) groups.  Consequently, we show that~\(\G[H]\) can be constructed by contracting \(\textup{U}_{q}(2)\) with respect to the same subgroup~\(\T\).
 
\medskip 

Let us briefly outline the structure of this article. We have gathered all the necessary preliminaries in Section~\ref{sec:prelim}. In general, the category of Yetter\nb-Drinfeld represenations of a~\(\Cst\)\nb-quantum group~\(\G\), denoted by~\(\YDCorepcat(\G)\), is a braided monoidal category, and additionally the braiding isomorphisms are unitary operators. In short, we call such a category as a \emph{unitarily braided monoidal category}. Motivated by~\cite{W1996}, a general theory of braided \(\Cst\)\nb-quantum groups~\cite{R2016} is developed using manageable braided multiplicative unitaries~\cites{R2013, MRW2017} in~\(\YDCorepcat(\G)\), where~\(\G\) is a regular \(\Cst\)\nb-quantum group. 

Suppose \(\Qgrp{G}{A}\) is a regular quasitriangular \(\Cst\)\nb-quantum group with respect to a unitary \(\Rmattxt\)\nb-matrix \(\Rmat\in\U(\hat{A}\otimes \hat{A})\), where~\(\DuQgrp{G}{A}\) is the dual of~\(\G\). Then the representation category~\(\Corepcat(\G)\) of~\(\G\) is a unitarily braided monoidal category~\cite{MRW2016}*{Section 3}. Therefore, it is natural to work with manageable braided multiplicative unitaries in \(\Corepcat(\G)\) instead of~\(\YDCorepcat(\G)\) whenever \(\G\) is a quasitriangular quantum group and construct braided \(\Cst\)\nb-quantum groups out of them. This is done in Section~\ref{sec:Bcst-qsi} using the results presented in the Appendix~\ref{sec:YD-rep} dealing with the Yetter\nb-Drinfeld representations of~\(\G\). In particular, the construction of braided \(\Cst\)\nb-quantum groups over~\(\G\) is presented in Theorem~\ref{the:BQgrp-quasitriag}. 

From Section~\ref{sec:bcstqgpT} onwards we fix~\(\G\) to be the circle group~\(\T\) 
(viewed as a quantum group) and the \(\Rmattxt\)\nb-matrix~\(\Rmat\) defined in~\eqref{eq:bichar}. Let \(\Hils[L]\) be a Hilbert space equipped with an orthonormal basis \(\{e_{ij}\}_{i,j\in\Z}\). 
A concrete realisation of~\eqref{eq:gen-E2} on~\(\Hils[L]\) is given by
\begin{equation}
 \label{eq:con-gen-E2}
  ve_{i,j}\defeq e_{i-1,j},
  \quad
  ne_{i,j}\defeq q^{i}e_{i,j+1}.
\end{equation}
In fact, any Hilbert space realisation of~\eqref{eq:gen-E2} is either one dimensional 
or infinite dimensional, and the direct integral of all infinite dimensional irreducible representations of~\eqref{eq:gen-E2} is unitarily equivalent to~\eqref{eq:con-gen-E2}, see~\cite{W1991b}*{Section 3 (B)}. 
Therefore, \eqref{eq:con-gen-E2} defines a faithful nondegenerate \Star{}representation 
of~\(B\) on~\(\Hils[L]\).

We identify~\(\Hils[L]\cong\ell^{2}(\Z)\otimes\ell^{2}(\Z)\) and use the canonical representation  of~\(\T\) on the second tensor factor. This makes~\(\Hils[L]\) an object of~\(\Corepcat(\T)\) and~\(\Rmat\) induces the unitary braiding~\(\Braiding{}{}\in\U(\Hils[L]\otimes\Hils[L])\). On standard basis elements~\(\{e_{i,j}\otimes e_{k,l}\}_{i,j,k,l\in\Z}\) of~\(\Hils[L]\otimes\Hils[L]\) the action of~\(\Braiding{}{}\) is given by \(e_{i,j}\otimes e_{k,l}\to \zeta^{-jl} e_{k,l}\otimes e_{i,j}\). Starting with these data, we construct a manageable braided multiplicative unitary~\(\BrMultunit\in\U(\Hils[L]\otimes\Hils[L])\) in the category~\(\Corepcat(\T)\), see Theorem~\ref{theorem:main}. 

The main results of this article are presented in Section~\ref{sec:Eq2C*-algebra}. We apply Theorem~\ref{the:BQgrp-quasitriag} to~\(\BrMultunit\) and construct~\(\qE=\Bialg{B}\) as a braided group over~\(\T\) from it in Theorem~\ref{the:BrdqE2}. 
The associated bosonisation~\(\Qgrp{H}{C}\) is described in Theorem~\ref{the:boson}. 
In the final section~\ref{sec:contraction}, we present the contraction procedure 
between~\(\G[H]\) and~\(\textup{U}_{q}(2)\) group. 

\section{Preliminaries}
\label{sec:prelim}
All Hilbert spaces and \(\Cst\)\nb-algebras (which are not explicitly multiplier algebras) are assumed to be separable. For a \(\Cst\)\nb-algebra~\(A\), 
let \(\Mult(A)\) be its multiplier algebra and let \(\U(A)\) be the group of unitary multipliers of~\(A\).  The unit of~\(\Mult(A)\) is denoted by 
\(1_{A}\). For two norm closed subsets \(X\) and~\(Y\) of a 
\(\Cst\)\nb-algebra~\(A\) and~\(T\in\Mult(A)\), set 
\[
  XTY\defeq\{xTy: x\in X, y\in Y\}^\CLS .
\]

Let~\(\Cstcat\) be the category of \(\Cst\)\nb-algebras with
nondegenerate \Star{}homomorphisms \(\varphi\colon A\to\Mult(B)\) as
morphisms \(A\to B\); let \(\Mor(A,B)\) denote the set of morphisms. 
We use the same symbol for an element of~\(\Mor(A,B)\) and its 
unique unital extension from~\(\Mult(A)\) to~\(\Mult(B)\).

Let~\(\Hils\) be a Hilbert space.  A \emph{representation} of a
\(\Cst\)\nb-algebra~\(A\) is a nondegenerate
\Star{}homomorphism \(\pi\colon A\to\Bound(\Hils)\).  Since
\(\Bound(\Hils)=\Mult(\Comp(\Hils))\) and the nondegeneracy
conditions \(\pi(A) \Comp(\Hils)=\Comp(\Hils)\) and
\(\pi(A) \Hils=\Hils\) are equivalent; hence \(\pi\in\Mor(A,\Comp(\Hils))\). 
The identity operator on~\(\Hils\) is denoted by~\(\Id_{\Hils}\). Whereas 
the the unit element of~\(\Mult(\Comp(\Hils))=\Bound(\Hils)\) 
is denoted by~\(1_{\Hils}\).

We write~\(\Flip\) for the tensor flip \(\Hils\otimes\Hils[K]\to
\Hils[K]\otimes\Hils\), \(x\otimes y\mapsto y\otimes x\), for two
Hilbert spaces \(\Hils\) and~\(\Hils[K]\).  We write~\(\flip\) for the
tensor flip isomorphism \(A\otimes B\to B\otimes A\) for two
\(\Cst\)\nb-algebras \(A\) and~\(B\).

\subsection{C*-quantum groups, actions and representations}
%\subsection{Multiplicative unitaries}
\label{sec:multunit_quantum_groups} 
 Let~\(\Hils\) be a Hilbert space. An element  
 \(\Multunit\in\U(\Hils\otimes\Hils)\) is called a 
 \emph{multiplicative unitary}~\cite{BS1993}*{Definition 1.1} if it satisfies the 
 \emph{pentagon equation}
  \begin{equation}
    \label{eq:pentagon}
    \Multunit_{23}\Multunit_{12}
    = \Multunit_{12}\Multunit_{13}\Multunit_{23}
    \qquad
    \text{in \(\U(\Hils\otimes\Hils\otimes\Hils).\)}
  \end{equation}
Furthermore, ~\(\Multunit\in\U(\Hils\otimes\Hils)\) 
is said to be \emph{manageable} if there is a strictly positive
operator~\(Q\) on~\(\Hils\)
and a unitary \(\widetilde{\Multunit}\in \U(\conj{\Hils}\otimes\Hils)\)
satisfying
\begin{equation}
  \label{eq:Multunit_manageable}
   \Multunit (Q\otimes Q)\Multunit^{*} = Q\otimes Q,
   \quad
  \langle x\otimes u\mid \Multunit\mid z\otimes y\rangle
  = \langle\conj{z}\otimes Q u\mid \widetilde{\Multunit} \mid
  \conj{x}\otimes Q^{-1}y\rangle
\end{equation}
 for all \(x,z\in\Hils\), \(u\in\dom(Q)\) and \(y\in\dom(Q^{-1})\)
(see \cite{W1996}*{Definition 1.2}).
Here~\(\conj{\Hils}\) is the conjugate Hilbert space, and an operator 
is \emph{strictly positive} if it is positive and self-adjoint with trivial kernel.
The condition \(\Multunit (Q\otimes Q)\Multunit^{*} = Q\otimes Q\)
means that the unitary~\(\Multunit\)
commutes with the unbounded operator~\(Q\otimes Q\).

A \emph{\(\Cst\)\nb-quantum group}~\(\G=\Bialg{A}\) consists of a \(\Cst\)\nb-algebra 
\(A\) and an element~\(\Comult[A]\in\Mor(A,A\otimes A)\) 
constructed from a  manageable multiplicative 
unitary \(\Multunit\)~\cite{W1996}*{Theorem 1.5}.  Let~\(\DuQgrp{G}{A}\) 
be the dual \(\Cst\)\nb-quantum group and let~\(\multunit\in\U(\hat{A}\otimes A)\) 
be the reduced bicharacter. In particular, 
\begin{equation}
  \label{eq:W-slices}
    \begin{aligned}
               A &\defeq\{(\omega\otimes\Id_{A})\multunit\mid \omega\in\hat{A}'\}^\CLS\subset\Bound(\Hils),\\
        \hat{A} &\defeq\{(\Id_{\hat{A}}\otimes\omega)\multunit\mid \omega\in A'\}^\CLS\subset\Bound(\Hils).
    \end{aligned}
\end{equation}
Moreover, \(\Comult[A]\) and~\(\DuComult[A]\) are characterised by 
\begin{equation}
  \label{eq:redbichar}
   \begin{aligned}
    (\Id_{\hat{A}}\otimes\Comult[A])\multunit &=\multunit_{12}\multunit_{13} \qquad\text{in~\(\U(\hat{A}\otimes A\otimes A)\),}\\
   (\DuComult[A]\otimes\Id_{A})\multunit      &=\multunit_{23}\multunit_{13} \qquad\text{in~\(\U(\hat{A}\otimes\hat{A}\otimes A)\).}
   \end{aligned}
\end{equation} 
We reserve the phrase ``quantum groups'' for \(\Cst\)\nb-quantum groups. 
An important concept in the theory of quantum groups is the concept of \emph{regularity} introduced by Baaj and Skandalis in~\cite{BS1993}. 
A quantum group~\(\Qgrp{G}{A}\) is said to be regular if ~\((\hat{A}\otimes 1_{A})\multunit (1_{\hat{A}}\otimes A)=\hat{A}\otimes A\). 
Dual of a regular quantum group is again regular.
 A \emph{\textup(right\textup) action} of~\(\Qgrp{G}{A}\) on a
  \(\Cst\)\nb-algebra~\(C\) is an injective morphism \(\delta\colon C\to C\otimes
  A\) with the following properties:
  \begin{enumerate}
    \item \(\delta\) is a comodule structure, that is,
    \begin{equation}
      \label{eq:right_action}
      (\Id_C\otimes\Comult[A])\circ\delta
      = (\delta\otimes\Id_A)\circ\delta;
    \end{equation}
  \item \(\delta\) satisfies the \emph{Podle\'s condition}:
    \begin{equation}
      \label{eq:Podles_cond}
      \delta(C)(1_C\otimes A)=C\otimes A.
    \end{equation}
 \end{enumerate}
  We call \((C,\delta)\) a \emph{\(\G\)\nb-\(\Cst\)\nb-algebra}.  We shall
  drop~\(\delta\) from our notation whenever it is clear from the context.
  A morphism \(f\colon C\to D\) between two
  \(\G\)\nb-\(\Cst\)\nb-algebras \((C,\delta_{C})\) and \((D,\delta_{D})\) is
  \emph{\(\G\)\nb-\hspace{0pt}equivariant} if \(\delta_{D}\circ f =
  (f\otimes\Id_A)\circ\delta_{C}\).  Let \(\Mor^{\G}(C,D)\) be the set of
  \(\G\)\nb-equivariant morphisms from~\(C\) to~\(D\).  Let
  \(\Cstcat(\G)\) be the category with \(\G\)\nb-\(\Cst\)-algebras as
  objects and \(\G\)\nb-equivariant morphisms as arrows.

  A (right) \emph{representation} of~\(\G\) on a \(\Cst\)\nb-algebra~\(D\) 
  is an element \(\corep{U}\in\U(D\otimes A)\)
  with
  \begin{equation}
    \label{eq:corep_cond}
    (\Id_{D}\otimes\Comult[A])\corep{U} =\corep{U}_{12}\corep{U}_{13}
    \qquad\text{in }\U(D\otimes A\otimes A).
  \end{equation}
In particular, if \(D=\Comp(\Hils[L])\) for some Hilbert space~\(\Hils[L]\), then \(\corep{U}\) is said to be a (right) 
representation of~\(\G\) on~\(\Hils[L]\).

The \emph{tensor product of representations}~\(\corep{U}^{i}\in\U(\Comp(\Hils[L]_{i})\otimes A)\) for~\(i=1,2\) 
is defined by~\(\corep{U}^{1}_{13}\corep{U}^{2}_{23}\in\U(\Comp(\Hils[L]_1\otimes\Hils[L]_2)\otimes A)\). It is  
denoted by \(\corep{U}^{1}\tenscorep \corep{U}^{2}\). An element \(t\in\Bound(\Hils[L]_1,\Hils[L]_2)\) is called
an \emph{intertwiner} if \((t\otimes 1_A)\corep{U}^1 =\corep{U}^2(t\otimes 1_A)\).  The set of all intertwiners
between \(\corep{U}^1\) and~\(\corep{U}^2\) is denoted by \(\Hom^{\G}(\corep{U}^1, \corep{U}^2)\). A routine computation 
shows that \(\tenscorep\) is associative; and the trivial \(1\)\nb-dimensional representation is a tensor unit. 
This gives representations a structure of \(\textup{W}^*\)\nb-category, which we denote by~\(\Corepcat(\G)\); 
see \cite{SW2007}*{Section 3.1--2} for more details. Objects of~\(\Corepcat(\G)\) are pairs~\((\Hils[L],\corep{U})\) 
consisting of a Hilbert space~\(\Hils[L]\) and a representation~\(\corep{U}\) of~\(\G\) on~\(\Hils[L]\).

  A \emph{covariant representation} of \((C,\delta,\G)\) on a Hilbert
  space~\(\Hils[L]\) is a pair~\((\corep{U},\varphi)\) consisting of a 
  representation \(\corep{U}\in\U(\Comp(\Hils[L])\otimes A)\) of~\(\G\) and 
  a representation \(\varphi\colon C\to\Bound(\Hils[L])\) of~\(C\) 
  that satisfy the covariance condition
  \begin{equation}
    \label{eq:covariant_corep}
    (\varphi\otimes\Id_A)\delta(c) =
    \corep{U}(\varphi(c)\otimes 1_A)\corep{U}^*
    \qquad\text{in }\U(\Comp(\Hils[L])\otimes A)
  \end{equation}
  for all \(c\in C\).  A covariant representation is called
  \emph{faithful} if~\(\varphi\) is faithful. Faithful covariant 
  representations always exists, 
  see for instance~\cite{MRW2014}*{Example~4.4}.
  
  \begin{example}
\label{ex:lcq-as-qg}
Let~\(G\) be a locally compact group. Let~\(\Hils\defeq L^{2}(G)\) with respect to the right Haar measure on~\(G\). 
A routine computation shows that the operator \((\Multunit \xi)(g_{1},g_{2})\defeq \xi (g_{1}g_{2},g_{2})\) 
for all~\(\xi\in L^{2}(G\times G)\) is a manageable multiplicative unitary with~\(Q=\Id_{\Hils}\) and~\(\widetilde{\Multunit}
=\Multunit^{*}\) and generates the quantum group~\(\Qgrp{G}{\Contvin(G)}\) where 
\((\Comult[\Contvin(G)]f)(g_{1},g_{2})\defeq f(g_1 g_2)\) for all~\(f\in\Contvin(G)\). In this way, \(G\) is viewed as a quantum group. 
%Every continuous action~\(\phi\) of~\(G\) on a~\(\Cst\)\nb-algebra~\(D\)  induces an 
%action~\(\delta_{D}\) of~\(\Qgrp{G}{\Contvin(G)}\) on~\(D\) 
%given by~\((\delta_{D}(d) f)(g)\defeq f(\phi_{g}(d))\) for all~\(f\in\Contvin(G)\), \(g\in G\) and vice versa. 
Moreover, the category \(\Cstcat(\G)\) is equivalent to the category \(\Cstcat(G)\) with 
\(G\)\nb-\(\Cst\)-algebras as objects and \(G\)\nb-equivariant morphisms as arrows. Similarly,~\(\Corepcat(\G)\) is also equivalent to the representation category~\(\Corepcat(G)\) of~\(G\). Let~\(\mu\) be the right regular representation of~\(G\) on~\(\Hils\). The dual of \(\G\) is~\(\DuG=\Bialg{\Cred(G)}\), where~\(\Comult[\Cred(G)](\mu_{g})\defeq \mu_{g}\otimes\mu_{g}\). 
Also, \(\G\) and~\(\DuG\) are regular quantum groups, and in addition, \(\DuG\) coincides with~\(\hat{G}\) as quantum group 
whenever \(G\) is Abelian. 

\end{example}
  
  For simplicity, whenever a locally compact group~\(G\) is viewed as a quantum group, 
  we use the same notation~\(G\) to denote corresponding quantum 
  group~\(\Qgrp{G}{\Contvin(G)}\). 
      
  \subsection{Quasitriangular quantum groups}
  \label{sec:Quasitriag}
   Let~\(\Qgrp{G}{A}\) be a~quantum group 
 and let~\(\DuQgrp{G}{A}\) be the dual. An element 
 \(\Rmat\in\U(\hat{A}\otimes\hat{A})\) is 
 said to be an~\(\Rmattxt\)\nb-matrix on~\(\DuG\) if it is a bicharacter: 
  \[
    (\Id_{\hat{A}}\otimes\DuComult[A])\Rmat=\Rmat_{12}\Rmat_{13}, 
    \quad 
  (\DuComult[A]\otimes\Id_{\hat{A}})\Rmat =\Rmat_{23}\Rmat_{13},
  \quad\text{in \(\U(\hat{A}\otimes\hat{A}\otimes\hat{A})\);}
  \]
  and satisfies the~\(\Rmattxt\)\nb-matrix condition:
 \begin{equation}
  \label{eq:R-mat}
  \Rmat(\flip\DuComult[A](\hat{a}))\Rmat^{*}=\DuComult[A](\hat{a}) 
  \qquad\text{for all~\(\hat{a}\in\hat{A}\).}
 \end{equation}
 A \emph{quasitriangular quantum group} is a quantum 
 group~\(\Qgrp{G}{A}\) with an \(\Rmattxt\)\nb-matrix \(\Rmat\in\U(\hat{A}\otimes\hat{A})\), 
 see~\cite{MRW2016}*{Definition 3.1}. 
 
  \begin{example}
  \label{ex:T-quasitriag}
   Every continuous bicharacter~\(\Z\times\Z\to\T\) satisfies the~\(\Rmattxt\)\nb-matrix 
  condition~\eqref{eq:R-mat} because \(\hat{A}=\Contvin(\Z)\) is a commutative 
  \(\Cst\)\nb-algebra. Thus~\(\T\) is quasitringular with respect to any 
  bicharacter on~\(\Z\times\Z\). In particular, 
  for~\(q\in\C\setminus\{0\}\) the bicharacter~\(\Rmat\colon\Z\times\Z\to\T\) 
  in~\eqref{eq:bichar} is an \(\Rmattxt\)\nb-matrix on~\(\Z\).
 \end{example}

Throughout this subsection we assume~\(\Qgrp{G}{A}\) is a quasitriangular quantum group 
 with an~\(\Rmattxt\)\nb-matrix~\(\Rmat\in\U(\hat{A}\otimes\hat{A})\). Then the 
 categories \(\Corepcat(\G)\) and~\(\Cstcat(\G)\) are of particular interest. 
More precisely, \cite{MRW2016}*{Proposition 3.2 \& Theorem 4.3} 
show that \(\Corepcat(\G)\) is a unitarily braided monoidal category and~\(\Cstcat(\G)\) is a monoidal category, whenever~\(\G\) is a quasitriangular quantum group. We recall the explicit construction of the unitary braiding on~\(\Corepcat(\G)\) and the monoidal product~\(\boxtimes_{\Rmat}\) on~\(\Cstcat(\G)\). The latter construction was motivated by~\cite{NV2010}.

Let~\((\alpha,\beta)\) be an~\(\Rmat\)\nb-Heisenberg pair acting on a Hilbert space~\(\Hils[L]\). 
More explicitly,~\(\alpha ,\beta\in\Mor(A,\Comp(\Hils[L]))\) and 
satisfy the commutation relation~\(\multunit_{1\alpha}
\multunit_{2\beta}=\multunit_{2\beta}\multunit_{1\alpha}\Rmat_{12}\) in 
\(\U(\hat{A}\otimes\hat{A}\otimes\Comp(\Hils[L]))\), where 
\(\multunit_{1\alpha}\defeq ((\Id_{\hat{A}}\otimes\alpha)\multunit)_{13}\) and 
\(\multunit_{1\beta}\defeq ((\Id_{\hat{A}}\otimes\beta)\multunit)_{23}\) 
in~\(\U(\hat{A}\otimes\hat{A}\otimes\Comp(\Hils[L]))\), 
see~\cite{MRW2014}*{Definition 3.1}. Existence of \(\Rmat\)\nb-Heisenberg pairs 
is guaranteed by~\cite{MRW2014}*{Lemma 3.8}. 

Suppose~\((\Hils[L]_{i},\corep{U}^{i})\) 
are objects in~\(\Corepcat(\G)\) for~\(i=1,2\). The proof of~\cite{MRW2014}*{Theorem 4.1} 
shows that there is a unique solution~\(Z\in\U(\Hils[L]_{1}\otimes\Hils[L]_{2})\), independent of the choice of the 
\(\Rmat\)\nb-Heisenberg pair~\((\alpha,\beta)\), of the following equation
\begin{equation}
 \label{eq:Z}
  \corep{U}^{1}_{1\alpha}\corep{U}^{2}_{2\beta}Z_{12}=\corep{U}^{2}_{2\beta}\corep{U}^{1}_{1\alpha}
  \qquad\text{in~\(\U(\Hils[L]_{1}\otimes\Hils[L]_{2}\otimes\Hils[L])\).}
\end{equation}
The braiding unitary is defined by~\(\Braiding{\Hils[L]_{1}}{\Hils[L]_{2}}\defeq Z\circ\Flip\in\Hom^{\G}
(\corep{U}^{2}\tenscorep\corep{U}^{1},\corep{U}^{1}\tenscorep\corep{U}^{2})\), 
see~\cite{MRW2016}*{Equation 3.2}. 

Suppose~\((C_{i},\delta_{i})\) are objects in~\(\Cstcat(\G)\) and 
\((\corep{U}^{i},\varphi_{i})\) are faithful covariant representations of 
\((C_{i},\delta_{i},\G)\) on~\(\Hils[L]_{i}\) for~\(i=1,2\). Clearly, 
\((\Hils[L]_{i},\corep{U}^{i})\) are objects in~\(\Corepcat(\G)\) for~\(i=1,2\) 
and let~\(\Braiding{\Hils[L]_{1}}{\Hils[L]_{2}}\) be the unitary braiding. 
Define \(j_{i}\in\Mor(C_{i},\Comp(\Hils[L]_{1}\otimes\Hils[L]_{2}))\) by 
\begin{equation}
 \label{eq:def_j1_j2_Rmat}
  j_{1}(c_1)\defeq \varphi_{1}(c_1)\otimes 1_{\Hils[L]_{2}}, 
  \qquad 
  j_{2}(c_2)\defeq \Braiding{\Hils[L]_{1}}{\Hils[L]_{2}}(\varphi_{2}(c_{2})\otimes 1_{\Hils[L]_{1}})\Braiding{\Hils[L]_{1}}{\Hils[L]_{2}}^{*},
\end{equation}
for all~\(c_{i}\in C_{i}\) and~\(i=1,2\). Then~\(C_{1}\boxtimes_{\Rmat} C_{2}\defeq j_{1}(C_1)j_{2}(C_2)\subseteq\Bound(\Hils[L]_{1}\otimes\Hils[L]_{2})\) is a~\(\Cst\)\nb-algebra~\cite{MRW2014}*{Theorem 4.6} and~\(\delta_{1}\bowtie\delta_{2}(j_{1}(c_1)j_{2}(c_2))\defeq (j_{1}\otimes\Id_{A})\delta_{1}(c_1)(j_{2}\otimes\Id_{A})
\delta_{2}(c_2)\) defines an action of~\(\G\) on~\(C_1\boxtimes_{\Rmat}C_2\)~\cite{MRW2016}*{Proposition 4.1}. 
Let~\((C'_{i},\delta'_{i})\) be objects in~\(\Cstcat(\G)\) and~\(f_{i}\in\Mor^{\G}(C_{i},
C'_{i})\) for~\(i=1,2\). The monoidal product~\(f_{1}\boxtimes_{\Rmat}f_{2}
\in\Mor(C_{1}\boxtimes_{\Rmat}C_{2},C'_{1}\boxtimes_{\Rmat}C'_{2})\) in~\(\Cstcat(\G)\) is defined by 
\[
  (f_{1}\boxtimes_{\Rmat}f_{2})\circ  j_{1}=j'_{1}\circ f_{1}, 
  \qquad
  (f_{1}\boxtimes_{\Rmat}f_{2})\circ j_{2} = j'_{2}\circ f_{2}, 
\]
where~\(j'_{i}\in\Mor(C'_{i},C'_{1}\boxtimes_{\Rmat}C'_{2})\) are the canonical morphisms for~\(i=1,2\).

\section{Braided C*-quantum groups over quasitriangular quantum groups}
 \label{sec:Bcst-qsi}
Let~\(\Qgrp{G}{A}\) be a quasitriangular quantum group with an~\(\Rmattxt\)\nb-matrix 
\(\Rmat\in\U(\hat{A}\otimes\hat{A})\).
\begin{definition}
 \label{def:Brmult-T}
 Let \((\Hils[L],\corep{U})\) be an object of \(\Corepcat(\G)\). Let 
 \(\Braiding{\Hils[L]}{\Hils[L]}=Z\circ\Flip\in\Hom^{\G}(\corep{U}\tenscorep\corep{U},\corep{U}\tenscorep\corep{U})\) be the 
 braiding unitary. A unitary~\(\BrMultunit\in\U(\Hils[L]\otimes\Hils[L])\) is said to be a \emph{braided multiplicative unitary
 over~\(\G\) relative to~\((\corep{U},\Rmat)\)} if 
 \begin{enumerate}
  \item \(\BrMultunit\in\Hom^{\G}(\corep{U}\tenscorep\corep{U},\corep{U}\tenscorep\corep{U})\): 
  \begin{equation}
   \label{eq:U-inv}
    \BrMultunit_{12}\corep{U}_{13}\corep{U}_{23}=\corep{U}_{13}\corep{U}_{23}\BrMultunit_{12} 
    \qquad\text{in~\(\U(\Comp(\Hils[L]\otimes\Hils[L])\otimes A)\);}
  \end{equation}
  \item \(\BrMultunit\) satisfies the braided pentagon equation:
  \begin{equation}
   \label{eq:Braided-pentagon}
    \BrMultunit_{23}\BrMultunit_{12}=\BrMultunit_{12}(\Braiding{\Hils[L]}{\Hils[L]}_{23})\BrMultunit_{12}(\Braiding{\Hils[L]}{\Hils[L]}_{23})^{*}\BrMultunit_{23}
    \qquad\text{in~\(\U(\Hils[L]\otimes\Hils[L]\otimes\Hils[L])\).}
  \end{equation}
 \end{enumerate} 
\end{definition}

Recall the bounded coinverse~\(R_{A}\) of~\(\G\), which is an involutive normal antiautomorphism of~\(A\), as in~\cite{W1996}*{Theorem 1.5 (4)}.  
Then~\(\corep{U}^{c}\defeq \corep{U}^{\transpose\otimes R_{A}}\in
\U(\Comp(\conj{\Hils[L]})\otimes A)\) is the contragradient representation of~\(\corep{U}\), where 
\(\transpose\) is the transposition defined by \(T^{\transpose}\bar{x}\defeq \overline{T^{*}x}\) for 
all~\(T\in\Bound(\Hils[L])\) (see~\cite{SW2007}*{Definition 11}). There is a unique 
element~\(\widetilde{Z}\in\U(\conj{\Hils[L]}\otimes\Hils[L])\) satisfying~\eqref{eq:Z} for the 
representations~\((\corep{U}^{c},\corep{U})\).

Suppose~\(\G\) is constructed from a manageable multiplicative unitary 
\(\Multunit\in\U(\Hils\otimes\Hils)\) and~\(Q\) be the strictly positive operator 
acting on~\(\Hils\) appearing in~\eqref{eq:Multunit_manageable}. Let~\(\pi\in\Mor(A,\Comp(\Hils))\) 
be the embedding in~\eqref{eq:W-slices} and let 
\(\Corep{U}\defeq(\Id_{\Comp(\Hils[L])}\otimes\pi)\corep{U}\in\U(\Hils[L]\otimes\Hils)\).
\begin{definition}
\label{definition:manageability}
A braided multiplicative unitary \(\BrMultunit\in\U(\Hils[L]\otimes\Hils[L])\) over~\(\G\) relative
 to~\((\Corep{U},\Rmat)\) is said to be \emph{manageable} if there is a strictly positive operator 
 \(Q_{\Hils[L]}\) on~\(\Hils[L]\) and a unitary~\(\widetilde{\BrMultunit}\in\U(\conj{\Hils[L]}\otimes\Hils[L])\) such that 
\[
      \BrMultunit (Q_{\Hils[L]}\otimes Q_{\Hils[L]})\BrMultunit^{*} 
   = Q_{\Hils[L]}\otimes Q_{\Hils[L]}, 
   \qquad 
      \Corep{U}(Q_{\Hils[L]}\otimes Q)\Corep{U}^{*}
   = Q_{\Hils[L]}\otimes Q,
 \]
 and
 \begin{equation}
  \label{eq:Br-mang}
   \langle x\otimes u \mid Z^{*}\BrMultunit\mid y\otimes v\rangle 
  =\langle\conj{y}\otimes Q_{\Hils[L]}u\mid \widetilde{\BrMultunit}\widetilde{Z}^{*}\mid \conj{x}\otimes Q_{\Hils[L]}^{-1}v\rangle
 \end{equation}  
  for all~\(x,y\in\Hils[L]\), \(u\in\dom(Q_{\Hils[L]})\), \(v\in\dom(Q_{\Hils[L]}^{-1})\). 
\end{definition}

The main result of this section is the construction of 
braided \(\Cst\)\nb-quantum groups from manageable braided multiplicative 
unitaries in~\(\Corepcat(\G)\).

\begin{theorem}
 \label{the:BQgrp-quasitriag}
 Let~\(\BrMultunit\in\U(\Hils[L]\otimes\Hils[L])\) be a manageable braided multiplicative unitary 
 over a regular quantum group~\(\Qgrp{G}{A}\) relative to~\((\Corep{U},\Rmat)\). Let 
 \begin{equation}
  \label{eq:def-B1904}
   B\defeq\{(\omega\otimes\Id_{\Comp(\Hils[L])})\BrMultunit\mid\omega\in\Bound(\Hils[L])_{*}\}^\CLS\subset\Bound(\Hils[L]).
 \end{equation}
 Then
 \begin{enumerate}
  \item \(B\) is a nondegenerate, separable \(\Cst\)\nb-subalgebra of~\(\Bound(\Hils[L])\)\textup{;}
  \item Define~\(\beta(b)\defeq \corep{U}(b\otimes 1_{A})\corep{U}^{*}\) for all~\(b\in B\). Then 
  \(\beta\in\Mor(B,B\otimes A)\) and~\((B,\beta)\) is an object of~\(\Cstcat(\G)\)\textup{;}
  \item \(\BrMultunit\in\U(\Comp(\Hils[L])\otimes B)\)\textup{;}
 \suspend{enumerate}
Consider the twisted tensor product~\(B\boxtimes_{\Rmat}B\). 
  Suppose~\(j_{1},j_{2}\in\Mor^{\G}(B,B\boxtimes_{\Rmat}B)\) are the canonical elements. 
 \resume{enumerate}  
  \item There exists a unique~\(\Comult[B]\in\Mor^{\G}(B,B\boxtimes_{\Rmat}B)\) characterised 
  by 
  \begin{equation}
   \label{eq:ComulB-R}
      (\Id_{\Comp(\Hils[L])}\otimes\Comult[B])\BrMultunit=\bigl((\Id_{\Comp(\Hils[L])}\otimes j_{1})\BrMultunit\bigr) 
      \bigl((\Id_{\Comp(\Hils[L])}\otimes j_{2})\BrMultunit\bigr).
   \end{equation}
   Moreover, \(\Comult[B]\) is coassociative: \((\Comult[B]\boxtimes_{R}\Id_{B})\circ\Comult[B]=(\Id_{B}\boxtimes_{\Rmat}\Comult[B])\circ\Comult[B]\)
   and satisfies the cancellation conditions: \(\Comult[B](B)j_{1}(B)=B\boxtimes_{\Rmat}B=\Comult[B](B)j_{2}(B)\).
 \end{enumerate} 
 The pair~\(\Bialg{B}\) is said to be the \emph{braided \(\Cst\)\nb-quantum group} \textup{(}over~\(\G\)\textup{)} generated 
by the braided multiplicative unitary~\(\BrMultunit\).
\end{theorem}

\begin{proof}
Let~\((\Hils[L],\corep{V})\) be the object in~\(\Corepcat(\DuG)\) induced by 
the~\(\Rmattxt\)\nb-matrix~\(\Rmat\) from~\((\Hils[L],\corep{U})\) in~\eqref{eq:du-corep-R-mat}. 
The functor~\(\mathcal{F}\colon\Corepcat(\G)\to\YDCorepcat(\G)\) in the Proposition~\ref{prop:YD-vs-Corep-R} 
does not change the underlying Hilbert spaces and the morphisms.  
Since \(\BrMultunit\in\Hom^{\G}(\corep{U}\tenscorep\corep{U},\corep{U}\tenscorep\corep{U})\) in~\(\Corepcat(\G)\) it is also an endomorphism 
of the object~\((\Hils[L],\corep{U},\corep{V})\) in~\(\YDCorepcat(\G)\); hence 
\(\BrMultunit\in\Hom^{\DuG}(\corep{V}\tenscorep\corep{V},\corep{V}\tenscorep\corep{V})\) in~\(\Corepcat(\DuG)\). This implies that \eqref{eq:Braided-pentagon} remains same 
in~\(\YDCorepcat(\G)\). Hence, \(\BrMultunit\) is also a braided multiplicative unitary in the sense of 
\cite{MRW2017}*{Definition 3.2}. Let~\(\pi\in\Mor(A,\Comp(\Hils))\) and~\(\hat{\pi}\in\Mor(\hat{A},
\Comp(\Hils))\) be the canonical embeddings in~\eqref{eq:W-slices}. Define~\(\corep{R}'\defeq 
(\hat{\pi}\otimes\hat{\pi})\Rmat\in\U(\Hils\otimes\Hils)\) and~\(\Corep{V}\defeq (\Id_{\Comp(\Hils[L])}\otimes\hat{\pi})\corep{V}\in\U(\Hils[L]\otimes\Hils)\). Using~\cite{MRW2012}*{Equation 33} we obtain a concrete 
realisation of~\eqref{eq:du-corep-R-mat} as follows
\[
  \corep{R}'_{23}\Corep{U}_{12}\corep{R}'{}^{*}_{23}=\Corep{U}_{12}\Corep{V}_{13}
  \qquad\text{in~\(\U(\Hils[L]\otimes\Hils\otimes\Hils)\).}
\]
 Manageability of~\(\BrMultunit\) implies~\(\Corep{U}\) commutes with~\(Q_{\Hils[L]}\otimes Q\). Since ~\(\corep{R}'\) is a bicharacter it commutes with~\(Q\otimes Q\) by~\cite{MRW2012}*{Proposition 3.10}.
 Therefore, \(\Corep{V}_{13}\) commutes with~\(Q_{\Hils[L]}\otimes Q\otimes Q\) 
 which implies that~\(\Corep{V}\) commutes with~\(Q_{\Hils[L]}\otimes Q\). Thus~\(\BrMultunit\) is also 
 manageable in the sense of~\cite{MRW2017}*{Definition 3.5}. Hence \cite{R2016}*{Theorem 5.1} 
 completes the proof. In fact, the conclusions~\((1)-(3)\) and, assuming the existence 
 of~\(\Comult[B]\) characterised by~\eqref{eq:ComulB-R}, the cancellation conditions in~\((4)\) 
 follow immediately. 
 Define \(\Comult[B](b)\defeq \BrMultunit (b\otimes 1_{\Hils[L]})\BrMultunit^{*}\) for all~\(b\in B\). 
  Observe that \((B\hookrightarrow\Bound(\Hils[L]),\corep{U})\) is a faithful covariant 
 representation of~\((B,\beta,\G)\) on~\(\Hils[L]\). Hence, we use the braiding 
 operator~\(\Braiding{}{}\) in Definition~\ref{def:Brmult-T} to define the canonical embeddings~\(j_{1},j_{2}\in \Mor^{\G}(B,B\boxtimes_{\Rmat}B)\) in~\eqref{eq:def_j1_j2_Rmat}. Then braided pentagon 
 equation~\eqref{eq:Braided-pentagon} and~\((1)\) ensures that~\(\Comult[B]\colon B\to\Mult(B\boxtimes_{\Rmat} B)\)
 is a~\(\G\)\nb-equivariant \Star{}homomorphism, satisfies~\eqref{eq:ComulB-R} and is coassociative. Finally, 
 \(\Comult[B](B)(B\boxtimes_{\Rmat}B)=\Comult[B](B)j_{1}(B)j_{2}(B)=(B\boxtimes_{\Rmat}B) j_{2}(B)
 =B\boxtimes_{\Rmat}B\) shows that~\(\Comult[B]\) is nondegenerate.
\end{proof}

\section{Braided multiplicative unitary of $\textup{E}_{q}(2)$ groups}
 \label{sec:bcstqgpT}
Let~\(\{e_{p}\}_{p\in\Z}\) be an orthonormal basis of~\(\Hils=\ell^{2}(\Z)\). 
The operator~\(\hat{N}e_{p}\defeq pe_{p}\) is an unbounded densely defined self adjoint operator 
with \(\Sp(\hat{N})=\Z\) and generates the \(\Cst\)\nb-algebra~\(\Contvin(\Z)\). The 
operator~\(ze_{p}\defeq e_{p+1}\) is unitary and generates the~\(\Cst\)\nb-algebra~\(\Cont(\T)\). 
A simple computation shows that~\(z\) and \(\hat{N}\) satisfy the following commutation relation 
\begin{equation}
 \label{eq:comm_z_hatN}
  z^{*}\hat{N}z=\hat{N}+1_{\Hils}.
\end{equation}
Define~\(\Multunit\in\U(\Hils\otimes\Hils)\) by
\[
  \Multunit\defeq (1_{\Hils}\otimes z)^{\hat{N}\otimes 1_{\Hils}}=\int_{\Z\times\T}x^{s} dE_{\hat{N}}(s)\otimes dE_{z}(x),
\]
where~\(dE_{\hat{N}}\) and \(dE_{z}\) are the spectral measures of~\(\hat{N}\) and~\(z\), respectively. The 
commutation relation~\eqref{eq:comm_z_hatN} gives the action of~\(\Multunit\) on the orthonormal 
basis~\(\{e_{k}\otimes e_{l}\}_{k,l\in\Z}\) by~\(e_{k}\otimes e_{l}\to e_{k}\otimes e_{l+k}\).  It is easy to 
verify that \(\Multunit\) is a manageable multiplicative unitary with~\(Q=\Id_{\Hils}\) 
and~\(\widetilde{\Multunit}(\bar{e}_{k}\otimes e_{l})\defeq \bar{e}_{k}\otimes e_{l+k}\). 
\(\Multunit\) generates~\(\T\):
\begin{align*}
  \Cont(\T) &=\{(\omega\otimes\Id_{\Comp(\Hils)})\Multunit\mid \omega\in\Bound(\Hils)_{*}\}^{\CLS},\\
  \Comult[\Cont(\T)](z) &=\Multunit (z\otimes 1_{\Hils})\Multunit^{*}=z\otimes z.
\end{align*}
The dual of~\(\T\) is~\(\Z\) with the comultiplication map  
\(\Comult[\Contvin(\Z)](\hat{N})=\hat{N}\otimes 1_{\Contvin(\Z)}\dotplus 
1_{\Contvin(\Z)}\otimes\hat{N}\).

 Recall the \(\Rmattxt\)\nb-matrix~\(\Rmat\colon\Z\times\Z\to\T\) 
defined in~\eqref{eq:bichar}. Then~\(\T\) is a 
 quasitriangular with respect to~\(\Rmat\) and 
\(\Corepcat(\T)\) a unitarily braided monoidal category.

Let~\(\Hils[L]\defeq\Hils\otimes\Hils\) and fix an orthonormal 
basis~\(\{e_{i,j}\defeq e_{i}\otimes e_{j}\}_{i,j\in\Z}\) of~\(\Hils[L]\).
Define~\(\Corep{U}\defeq 1_{\Hils}\otimes\Multunit\in\U(\Hils^{\otimes 3})\cong\U(\Hils[L]\otimes\Hils)\). Clearly, 
\(\Corep{U}(e_{i,j}\otimes e_{p})=e_{i,j}\otimes e_{p+j}\) and 
denote~\(\Corep{U}\) by~\(\corep{U}\) while viewed as an element of 
\(\U(\Comp(\Hils[L])\otimes\Cont(\T))\). Also, the first equation in \eqref{eq:redbichar} shows that \(\corep{U}\) is a representation of~\(\T\) on~\(\Hils[L]\). Hence \((\Hils[L],\corep{U})\) is an object of~\(\Corepcat(\T)\).
 
 Define~\(\alpha,\beta\in\Mor(\Cont(\T),\Comp(\Hils))\)  by~\(\alpha(z)\defeq z\) and~\(\beta(z)\defeq \widetilde{V}\), where \(\widetilde{V}\) is the unitary operator defined by~\(\widetilde{V}e_{p}=\zeta^{-p}e_{p}\), respectively. Then~\((\alpha,\beta)\) is an 
 \(\Rmat\)\nb-Heisenberg pair on~\(\Hils\) because~\(z\) and~\(\widetilde{V}\) commute up to~\(\zeta\):
 \[
   z\widetilde{V}e_{p}
   =\zeta^{-p}z e_{p}
   =\zeta^{-p}e_{p+1}
   =\zeta\widetilde{V}e_{p+1}
   =\zeta\widetilde{V}z e_{p}.
 \]
 In fact, \(z\) and~\(\widetilde{V}\) generate noncommutative two torus.
Using these we compute the unitary~\(Z\in\U(\Hils[L]\otimes\Hils[L])\) in~\eqref{eq:Z}. 
Equivalently, \(Z\) is uniquely defined by 
\begin{equation}
\label{eqn:expr_z}
Z_{12}=\corep{U}_{2\beta}^{*}\corep{U}_{1\alpha}^*\corep{U}_{2\beta}\corep{U}_{1\alpha} 
\qquad\text{in~\(\U(\Hils[L]\otimes\Hils[L]\otimes\Hils)\),}
\end{equation}
where \(\corep{U}_{1\alpha}=\corep{U}\) 
and \(\corep{U}_{1\beta}=(1_{\Hils}\otimes 1_{\Hils}\otimes \widetilde{V})^{(1_{\Hils}\otimes\hat{N}\otimes 1_{\Hils})}\). 
Clearly, on the basis elements we have 
\(\corep{U}_{1\beta}(e_{k,l}\otimes e_p)=\zeta^{-pl}e_{k,l}\otimes e_{p}\). Therefore, 
\begin{align*}
  \corep{U}_{2\beta}^{*}\corep{U}_{1\alpha}^*\corep{U}_{2\beta}\corep{U}_{1\alpha}(e_{i,j}\otimes e_{k,l}\otimes e_p)
&=\zeta^{-(p+j)l}\corep{U}_{2\beta}^{*}\corep{U}_{1\alpha}^*(e_{i,j}\otimes e_{k,l}\otimes e_{p+j})\\
&= \zeta^{-(p+j)l+pl}e_{i,j}\otimes e_{k,l}\otimes e_{p}
  =\zeta^{-jl}e_{i,j}\otimes e_{k,l}\otimes e_{p}.
\end{align*}
Hence the the unitary braiding operator \(\Braiding{}{}\defeq Z\circ \Flip\in\U(\Hils[L]\otimes\Hils[L])\) is given by
\begin{equation}
 \label{eq:braiding}
 Z (e_{i,j}\otimes e_{k,l}) =\zeta^{-jl}e_{i,j}\otimes e_{k,l},
 \qquad
\Braiding{}{}e_{i,j}\otimes e_{k,l}=\zeta^{-jl}e_{k,l}\otimes e_{i,j}.
\end{equation}

Define a closed operator \(X=\Mod{X}\Ph{X}\) on~\(\Hils[L]\otimes\Hils[L]\) by 
\begin{equation}
 \label{eq:def_X}
 \begin{aligned}
   \Mod{X}e_{i,j}\otimes e_{k,l} &=|q|^{k-i+1}e_{i,j}\otimes e_{k,l},\\
   \Ph{X}e_{i,j}\otimes e_{k,l}   &=\zeta^{-j}\Ph{q}^{k-i+1}e_{i-1,j-1}\otimes e_{k-1,l+1},\\
   X e_{i,j}\otimes e_{k,l}          &=\zeta^{-j}q^{k-i+1} e_{i-1,j-1}\otimes e_{k-1,l+1}.
\end{aligned}
\end{equation}
Then~\(X\) is a normal operator because~\(\Mod{X}\) commutes with its phase 
\(\Ph{X}\) in the polar decomposition and~\(\Sp(X)=\cl{\C}^{|q|}\). 

 Recall that the operator \(n\), defined by \(n e_{i,j}=q^{i}e_{i,j+1}\), is closed and injective. Hence it is invertible and on the standard basis elements 
 we have \(n^{-1}e_{i,j}=q^{-i}e_{i,j-1}\). Define 
 \(P\in\U(\Hils[L])\) by
 \begin{equation} 
  \label{eq:P-def}
  Pe_{i,j}=\zeta^{-j}e_{i,j}.
 \end{equation}
 Then the following computation shows~\(X=n^{-1}vP\otimes vn\):
\[
    (n^{-1}vP\otimes vn) e_{i,j}\otimes e_{k,l} 
 = \zeta^{-j} q^{k}(n^{-1}\otimes v) e_{i-1,j}\otimes e_{k,l+1}
 =\zeta^{-j}q^{k-i+1} e_{i-1,j-1}\otimes e_{k-1,l+1}.
\]
The quantum exponential function~\(\QuantExp\colon \cl{\C}^{|q|}\to \T\) is defined in~\cite{W1992a}*{Equation 1.2} by
 \begin{equation*}
    \QuantExp(z)=
    \begin{cases}
      \prod_{k=0}^{\infty}\frac{1+{|q|}^{2k}\overline{z}}{1+{|q|}^{2k}z} & \text{if}\ z\in\cl{\C}^{|q|}\setminus\{-{|q|}^{-2k}\mid k=0,1,\cdots\}, \\
      -1 & \text{otherwise}
    \end{cases}
  \end{equation*}
  Since~\(\QuantExp\) is a unitary multiplier of~\(\Contvin(\cl{\C}^{|q|})\), we observe  
  \(\QuantExp(X)\in\U(\Hils[L]\otimes\Hils[L])\).

\begin{theorem}
\label{theorem:main}
Let~\(\Corep{Y}=\Multunit_{13}\Multunit_{23}\in\U(\Hils\otimes\Hils\otimes\Hils\otimes\Hils)\). 
Then the operator \(\BrMultunit:=\QuantExp(X)\Corep{Y}\) 
is a manageable braided multiplicative unitary on~\(\Hils[L]\otimes\Hils[L]\) 
over~\(\T\) relative to~\((\corep{U},\Rmat)\).
\end{theorem}
The rest of this section is devoted to the proof of this theorem.
\begin{lemma}
\label{lem:pentagon}
Define~\(T(\lambda)\defeq\QuantExp(\lambda X)\) for all~\(\lambda\in\cl{\C}^{|q|}\). Then
\[
\QuantExp(X)_{23}T(\lambda)_{12}
=T(\lambda)_{12}\QuantExp(\lambda n^{-1}vP\otimes v^2P\otimes vn) \QuantExp(X)_{23}
\quad\text{in~\(\U(\Hils[L]\otimes\Hils[L]\otimes\Hils[L])\).}
\]
\end{lemma}

\begin{proof}
 Clearly, the equality holds for~\(\lambda=0\). Therefore, we fix a nonzero 
 element \(\lambda \in\cl{\C}^{|q|}\). 
 The operators~\(R=\lambda n^{-1}vP\otimes vn\otimes 1_{\Hils[L]}=\lambda X_{12}\) and~\(S=\lambda n^{-1}vP\otimes v^{2}P\otimes vn=\lambda (1_{\Hils[L]} \otimes v^{2}P\otimes 1_{\Hils[L]})X_{13}\) 
 are normal, each pair of operators \((\Mod{R},\Mod{S})\) and \((\Ph{R},\Ph{S})\) strongly commute and
 \[
    \Ph{R}\Mod{S}\Ph{R}^*=|q|^{-1}\Mod{S}, 
    \quad 
    \Ph{S}\Mod{R}\Ph{S}^*=|q|\Mod{R},
    \quad 
    \Sp(R),\Sp(S)= \cl{\C}^{|q|}.
 \]
 Also \(R^{-1}S=1_{\Hils[L]}\otimes n^{-1}vP\otimes vn=X_{23}\) is normal with spectrum \(\cl{\C}^{|q|}\). 
 Then \cite{W1992a}*{theorem  2.1-2.2} show that 
 \(R\dotplus S\) is normal with spectrum~\(\cl{\C}^{|q|}\) and
 \[
  \QuantExp(R^{-1}S)R\QuantExp(R^{-1}S)^{*}
  =R\dotplus S.
 \]
 Since the functional calculus is compatible with conjugation by 
 unitaries, \cite{W1992a}*{Theorem 3.1} implies 
 \begin{align*}
     \QuantExp(R^{-1}S)\QuantExp(R)\QuantExp(R^{-1}S)^{*}
 &=\QuantExp(\QuantExp(R^{-1}S)R\QuantExp(R^{-1}S)^{*})\\
 &=\QuantExp(R\dotplus S)
   =\QuantExp(R)\QuantExp(S). \qedhere
 \end{align*}
\end{proof}

For~\(\lambda\in\Specn\) define \(\BrMultunit^{\lambda}\defeq \QuantExp(\lambda X)\Corep{Y}\in
\U(\Hils[L]\otimes\Hils[L])\). 

\begin{proposition}
 \label{prop:BrMultunit} 
 The family of unitaries~\(\{\BrMultunit^{\lambda}\}_{\lambda\in\Specn}\) commute 
 with~\(\Corep{U}\tenscorep\Corep{U}\) and satisfy a variant of the braided pentagon 
 equation~\eqref{eq:Braided-pentagon}:
 \begin{equation}
  \label{eq:corep-family}
   \BrMultunit_{23}\BrMultunit^{\lambda}_{12}
   =\BrMultunit^{\lambda}_{12}\Braiding{}{}_{23}\BrMultunit^{\lambda}_{12}\Braiding{}{}_{23}^{*}\BrMultunit_{23} 
   \quad\text{in~\(\U(\Hils[L]\otimes\Hils[L]\otimes\Hils[L])\).}
 \end{equation}
\end{proposition}
\begin{proof}
The unitaries \(\Multunit_{25}\Multunit_{45}\) and 
\(\Multunit_{13}\Multunit_{23}\) commute in~\(\U(\Hils^{\otimes 5})\). 
Identifying~\(\Hils[L]=\Hils\otimes\Hils\) and~\(\Hils[L]\otimes\Hils[L]\otimes\Hils=\Hils^{\otimes 5}\) we 
get \(\Corep{U}\tenscorep\Corep{U}\) and~\(\Corep{Y}_{12}\) commutes in~\(\U(\Hils[L]\otimes\Hils[L]
\otimes\Hils)\). Moreover \(\Corep{U}\tenscorep\Corep{U}
(e_{i,j}\otimes e_{k,l}\otimes e_{p})=e_{i,j}\otimes e_{k,l}\otimes e_{p+j+l}\) 
 shows the operators~\(\Mod{X}_{12}\) and~\((\Ph{X})_{12}\) defined in~\eqref{eq:def_X} commute with 
 \(\Corep{U}\tenscorep\Corep{U}\). Therefore \(\Corep{U}\tenscorep\Corep{U}\) commutes with~\(\lambda X_{12}\) for 
 all~\(\lambda\in\Specn\). Therefore, \(\Corep{U}\tenscorep\Corep{U}\) commutes 
 with \(\QuantExp(\lambda X)_{12}\) and consequently with 
 \(\BrMultunit^{\lambda}_{12}\)  in~\(\U(\Hils[L]\otimes\Hils[L]\otimes\Hils)\).

 The action of~\(vn\) and \(\Corep{Y}=\Multunit_{13}\Multunit_{23}\) on the basis elements 
 of~\(\Hils[L]\) and~\(\Hils[L]\otimes\Hils[L]\) are given by~\(e_{i,j}\to q^{i} e_{i-1,j+1}\) and 
 \(e_{i,j}\otimes e_{k,l} \to e_{i,j}\otimes e_{k+i+j,l}\), respectively. Therefore, 
 \begin{align}
  \label{eq:vn-comm-V}
        \Corep{Y}(vn\otimes 1_{\Hils[L]})\Corep{Y}^{*} e_{i,j}\otimes e_{k,l}
   &= \Corep{Y}(vn\otimes 1_{\Hils[L]}) e_{i,j}\otimes e_{k-i-j,l} \nonumber \\
   &= q^{i}\Corep{Y} e_{i-1,j+1}\otimes e_{k-i-j,l} \nonumber \\
   &= q^{i} e_{i-1,j+1} \otimes e_{k,l} 
     =(vn\otimes 1_{\Hils[L]})e_{i,j}\otimes e_{k,l}
 \end{align}
  implies that~\(\Corep{Y}\) commutes with~\(vn\otimes 1_{\Hils[L]}\). This implies that 
  \eqref{eq:corep-family} is equivalent to
 \[ 
   \QuantExp( X)_{23}T(\lambda)_{12}\Corep{Y}_{23}\Corep{Y}_{12}\Corep{Y}_{23}^{*}
   \QuantExp(X)_{23}^{*}
   =T(\lambda)_{12}\Corep{Y}_{12}Z_{23}T(\lambda)_{13}\Corep{Y}_{13}Z^{*}_{23},
 \]
 where~\(\{T(\lambda)\}_{\lambda\in\cl{\C}^{|q|}}\) is the family of unitary operators defined 
 in Lemma~\ref{lem:pentagon}. 
 
 Now~\(\Corep{Y}\) is a multiplicative unitary because \(\Corep{Y}_{23}\Corep{Y}_{12}(e_{i,j}\otimes 
 e_{k,l}\otimes e_{s,t}) = e_{i,j}\otimes e_{k+i+j,l} \otimes e_{s+k+i+j+l,t} =
 \Corep{Y}_{12}\Corep{Y}_{13}\Corep{Y}_{23} (e_{i,j}\otimes e_{k,l}\otimes e_{s,t})\). This 
 simplifies the last equation to 
 \[ 
   \QuantExp(X)_{23}T(\lambda)_{12}\Corep{Y}_{12}\Corep{Y}_{13}
   \QuantExp(X)_{23}^{*}
   =T(\lambda)_{12}\Corep{Y}_{12}Z_{23}T(\lambda)_{13}\Corep{Y}_{13}Z^{*}_{23}.
 \]
   
 Using~\eqref{eq:def_X} we compute
 \begin{align*}	
  &   \Corep{Y}_{13}^{*}\Corep{Y}_{12}^{*}X_{23} \Corep{Y}_{12}\Corep{Y}_{13} (e_{i,j}\otimes e_{k,l}\otimes e_{s,t})\\
  &= \Corep{Y}_{13}^{*}\Corep{Y}_{12}^{*} X_{23} (e_{i,j}\otimes e_{k+i+j,l}\otimes e_{s+i+j,t})\\
  &= \zeta^{-l}q^{(s+i+j)-(k+i+j)+1} \Corep{Y}_{13}^{*}\Corep{Y}_{12}^{*}(e_{i,j}\otimes e_{k+i+j-1,l-1} \otimes e_{s+i+j-1,t+1})\\
  &= \zeta^{-l}q^{s-k+1} e_{i,j}\otimes e_{k-1,l-1} \otimes e_{s-1,t+1} 
    =X_{23} (e_{i,j}\otimes e_{k,l}\otimes e_{s,t}).
\end{align*}
 Therefore, \(\Corep{Y}_{12}\Corep{Y}_{13}\) commutes with~\(\QuantExp(X)_{23}\) and this implies 
 \[
  \QuantExp(X)_{23}T(\lambda)_{12}\QuantExp(X)_{23}^{*} \Corep{Y}_{12}\Corep{Y}_{13}
   =T(\lambda)_{12}\Corep{Y}_{12}Z_{23}T(\lambda)_{13}\Corep{Y}_{13}Z^{*}_{23}.
 \]
 Next we compute 
 \begin{align*}
  &   \Corep{Y}_{12}Z_{23}X_{13}Z_{23}^{*}\Corep{Y}_{12}^{*} (e_{i,j}\otimes e_{k,l}\otimes e_{s,t})\\ 
  & =\zeta^{lt}\Corep{Y}_{12}Z_{23}X_{13} (e_{i,j}\otimes e_{k-i-j,l}\otimes e_{s,t})\\ 
  &= \zeta^{lt} \zeta^{-j}q^{s-i+1}\Corep{Y}_{12}Z_{23} (e_{i-1,j-1}\otimes e_{k-i-j,l}\otimes e_{s-1,t+1})\\
  &=\zeta^{lt-j-l(t+1)}q^{s-i+1} e_{i-1,j-1}\otimes e_{k-2,l} \otimes e_{s-1, t+1}\\
  &= X_{13}(1_{\Hils[L]}\otimes v^{2}P\otimes 1_{\Hils[L]}) (e_{i,j}\otimes e_{k,l}\otimes e_{s,t}).
 \end{align*}
This implies 
\[
\QuantExp(X)_{23}T(\lambda)_{12}\QuantExp(X)_{23}^{*} \Corep{Y}_{12}\Corep{Y}_{13}
   =T(\lambda)_{12}\QuantExp(\lambda n^{-1}vP\otimes v^2P\otimes vn)\Corep{Y}_{12}Z_{23}\Corep{Y}_{13}Z^{*}_{23}.
\]
Finally, observe that~\(\Corep{Y}_{13}\) commutes with~\(Z_{23}\) and then the last equation follows from 
Lemma~\ref{lem:pentagon}.
\end{proof}

In particular, for \(\lambda=1\) the Proposition~\ref{prop:BrMultunit} shows that 
\(\BrMultunit\) is a braided multiplicative unitary over~\(\T\) 
with respect to~\((\Corep{U},\Rmat)\). 

It remains to show that~\(\BrMultunit\) is manageable. We start with the 
description of the operator~\(\widetilde{Z}\) appearing in the definition \ref{definition:manageability}. 
The contragradient \(\corep{U}^{c}=\corep{U}^{\transpose\otimes R}\) of \(\corep{U}\), 
where~\(R(z)\defeq z^{*}\) is the bounded coinverse of~\(\T\), acts on the basis 
elements~\(\conj{e_{i,j}}\otimes e_{p}\) as~\(\Corep{U}^{c} (\conj{e_{i,j}}\otimes e_p)
=\conj{e_{i,j}}\otimes e_{p-j}\). A similar calculation for \(\widetilde{Z}\) like that of \(Z\) yields 
\[
\widetilde{Z}\ \conj{e_{i,j}}\otimes e_{k,l}
=\zeta^{jl}\conj{e_{i,j}}\otimes e_{k,l}.
\] 

Next we define the operator \(Q_{\Hils[L]}\) required by definition~\ref{definition:manageability}:
\[
Q_{\Hils[L]}e_{i,j}=\abs{q}^{j}e_{i,j} .
\]
This is a strictly positive operator on~\(\Hils\) with spectrum \(\abs{q}^{\Z}\cup\{0\}\).
A simple calculation shows that \(\Corep{Y}\) and~\(X\) commute 
with \(Q_{\Hils[L]}\otimes Q_{\Hils[L]}\) and therefore \(\BrMultunit\) commutes with 
\(Q_{\Hils[L]}\otimes Q_{\Hils[L]}\). Also the unitary \(\Corep{U}\) leaves the first factor of the standard basis 
vector unchanged and~\(Q=\Id_{\Hils}\) implies \(\Corep{U}\) commutes with \(Q_{\Hils[L]}\otimes Q\).

Finally, we need a unitary~\(\widetilde{\BrMultunit}\in\U(\conj{\Hils[L]}\otimes\Hils[L])\) that satisfies 
\eqref{eq:Br-mang}. It is sufficient to check this if the involved vectors~\(x,y,u,v\) are 
standard basis vectors~\(x=e_{i,j}\), \(y=e_{s,t}\), \(u=e_{k,l}\) and~\(v=e_{a,b}\). 
Using the explicit formulas for~\(Z\), \(\widetilde{Z}\), \(\Corep{Y}\) and~\(Q\), we rewrite~\eqref{eq:Br-mang} as
\begin{equation}
 \label{eq:manag-aux-1}
    \zeta^{jl}\langle e_{i,j}\otimes e_{k,l} \mid \QuantExp(X)\mid e_{s,t}\otimes e_{a+s+t,b}\rangle 
  =\abs{q}^{l-b}\zeta^{-jb}\langle\conj{e_{s,t}}\otimes e_{k,l}\mid \widetilde{\BrMultunit}\mid \conj{e_{i,j}}\otimes e_{a,b}\rangle
\end{equation}
for all~\(a,b,i,j,k,l,s,t\in\Z\). To compute the left hand side of~\eqref{eq:manag-aux-1} we shall use the 
Fourier transform of \(\QuantExp\) restricted on the concentric circles~\(\abs{z}\in\abs{q}^{\Z}\)
\[
\QuantExp(z)=\sum_{m\in\Z}F_m(\abs{z})\Ph{z}^m,
\]
where \(z=\Ph{z}\Mod{z}\) and the coefficients \(F_m(\abs{q}^n)\) for~\(m,n\in\Z\) are real and satisfies
\begin{equation}
 \label{eqn:Fourier_trans}
F_m(\abs{q}^n)=(-\abs{q})^mF_{-m}(\abs{q}^{n-m}),
\end{equation}
see~\cite{B1992} or~\cite{W2001}*{Appendix A}. 

Now~\(e_{s,t}\otimes e_{a+s+t,b}\) is an eigenvector of~\(\Mod{X}\) with eigenvalue~\(\abs{q}^{a+t+1}\) and 
\(\Ph{X}^{m}\) acts on it by~\(e_{s,t}\otimes e_{a+s+t,b}\to \zeta^{-mt+\frac{m(m-1)}{2}}\Ph{q}^{m(a+t+1)} e_{s-m,t-m}\otimes e_{a+s+t-m,b+m}\).

Thus  
\begin{align*}
  &   \zeta^{jl}\langle e_{i,j}\otimes e_{k,l} \mid \QuantExp(X)\mid e_{s,t}\otimes e_{a+s+t,b}\rangle\\
  &= \sum_{m\in\Z} \zeta^{jl}\langle e_{i,j}\otimes e_{k,l} \mid \Ph{X}^m F_m(\Mod{X}) \mid e_{s,t}\otimes e_{a+s+t,b}\rangle \\
  &= \sum_{m\in\Z} \zeta^{jl-mt+\frac{m(m-1)}{2}}\Ph{q}^{m(a+t+1)} F_{m}(\abs{q}^{a+t+1})
        \delta_{i,s-m} \delta_{j,t-m} \delta_{k,a+s+t-m} \delta_{l,b+m}\\
  &=  \zeta^{jl-(l-b)t+\frac{(l-b)(l-b-1)}{2}}\Ph{q}^{(l-b)(a+t+1)}F_{l-b}(\abs{q}^{a+t+1})
        \delta_{i,s-l+b} \delta_{j,t-l+b} \delta_{k,a+s+t-l+b}\\
  &=  (-\abs{q})^{l-b} \zeta^{jl-(l-b)t+\frac{(l-b)(l-b-1)}{2}}\Ph{q}^{(l-b)(a+t+1)}F_{b-l}(\abs{q}^{a+t+1-l+b})\\
  &\qquad        \delta_{i,s-l+b} \delta_{j,t-l+b} \delta_{l-b,a+s+t-k}.
\end{align*}
The last expression is nonzero if and only if~\(l-b=s-i=t-j=a+s+t-k\). We also notice that 
\(\zeta=\Ph{q}^{2}\). These imply that \(a+s+t-k=k-a-i-j\), \(\delta_{l-b,a+s+t-k}=\delta_{l-b,k-a-i-j}
=\delta_{k+b-l,a+i+j}\) and \(\zeta^{jl-(l-b)t+\frac{(l-b)(l-b-1)}{2}}\Ph{q}^{(l-b)(a+t+1)}=\zeta^{jb}\Ph{q}^{(b-l)(s-k)}\). Therefore, 
the left hand side of~\eqref{eq:manag-aux-1} becomes
\begin{align*}
  &   \zeta^{jl}\langle e_{i,j}\otimes e_{k,l} \mid \QuantExp(X)\mid e_{s,t}\otimes e_{a+s+t,b}\rangle\\
  &= (-\abs{q})^{l-b} \zeta^{jb}\Ph{q}^{(b-l)(s-k)}F_{b-l}(\abs{q}^{k-s+1})\delta_{i,s-l+b} \delta_{j,t-l+b} \delta_{k+b-l,a+i+j}.
\end{align*}

Now we define an unbounded normal operator~\(\widetilde{X}=\Mod{\widetilde{X}}\Ph{\widetilde{X}}\) with 
spectrum~\(\Specn\) 
\[
\Mod{\widetilde{X}}(\conj{e_{ij}}\otimes e_{kl})\defeq |q|^{k-i+1}\ \conj{e_{ij}}\otimes e_{kl},
\quad
\Ph{\widetilde{X}}(\conj{e_{ij}}\otimes e_{kl})\defeq -\Ph{q}^{k-i}\ \conj{e_{i+1,\ j+1}}\otimes e_{k+1,\ l+1},
\]
and a unitary operator~\(\widetilde{\Corep{Y}} (\conj{e_{ij}}\otimes e_{kl})\defeq \conj{e_{ij}}\otimes e_{k+i+j,l}
\). We show that the unitary \(\widetilde{\BrMultunit}:=\QuantExp(\widetilde{X})^*\widetilde{Z}^2\widetilde{\Corep{Y}}\) 
 satisfies~\eqref{eq:manag-aux-1}:
\begin{align*}
 & \abs{q}^{l-b}\zeta^{-jb}\langle\conj{e_{s,t}}\otimes e_{k,l}\mid \widetilde{\BrMultunit}\mid \conj{e_{i,j}}\otimes e_{a,b}\rangle\\
 &= \abs{q}^{l-b}\zeta^{-jb}\langle\conj{e_{s,t}}\otimes e_{k,l}\mid \QuantExp(\widetilde{X})^*\widetilde{Z}^2\mid \conj{e_{i,j}}\otimes e_{a+i+j,b}\rangle \\
 &= \abs{q}^{l-b}\zeta^{jb}\langle\conj{e_{s,t}}\otimes e_{k,l}\mid \QuantExp(\widetilde{X})^*\mid \conj{e_{i,j}}\otimes e_{a+i+j,b}\rangle \\ 
 &= \sum_{m\in\Z} \abs{q}^{l-b}\zeta^{jb}\langle\conj{e_{s,t}}\otimes e_{k,l}\mid F_{m}(\Mod{\widetilde{X}})^{*} (\Ph{\widetilde{X}}^*)^{m} \mid \conj{e_{i,j}}\otimes e_{a+i+j,b}\rangle \\ 
 &=\sum_{m\in\Z} \abs{q}^{l-b}\zeta^{jb} F_{m}(\abs{q}^{k-s+1})\langle\conj{e_{s,t}}\otimes e_{k,l}\mid(\Ph{\widetilde{X}}^*)^{m} \mid \conj{e_{i,j}}\otimes e_{a+i+j,b}\rangle \\ 
 &=\sum_{m\in\Z} \abs{q}^{l-b}\zeta^{jb} F_{m}(\abs{q}^{k-s+1})(-1)^{m}\Ph{q}^{m(s-k)}\delta_{s+m,i}\delta_{t+m,j}\delta_{k+m,a+i+j}\delta_{l+m,b}\\
 &= (-\abs{q})^{l-b} \zeta^{jb}\Ph{q}^{(b-l)(s-k)} F_{b-l}(\abs{q}^{k-s+1})\delta_{i,s-l+b}\delta_{j,t-l+b}\delta_{k+b-l,a+i+j}.
\end{align*}

\section{Braided \(\textup{E}_{q}(2)\)-groups and bosonisation}
\label{sec:Eq2C*-algebra}
 Fix~\(q\in\C\) with~\(0<\Mod{q}<1\). Now we are going to describe the braided~\(\Cst\)\nb-quantum group 
 constructed from the manageable braided multiplicative 
 unitary~\(\BrMultunit\) in Theorem~\ref{theorem:main}. Hence we 
 shall use all the notations introduced in the previous section.

 Recall the~\(\Z\)\nb-action~\(\alpha\) on~\(\Contvin(\Specn)\) defined 
 by~\eqref{eq:Z-action}: \((\alpha_{m}f)(\lambda)\defeq f(q^{m}\lambda)\) for~\(m\in\Z\), \(f\in\Contvin(\Specn)\), 
 \(\lambda\in\Specn\). Let~\(\beta\) be the~\(\Z\)\nb-action on~\(\Contvin(\Specn)\) obtained by replacing \(q\) by~\(\Mod{q}\) in the 
 definition of~\(\alpha\)~\eqref{eq:Z-action}. Following a similar set of arguments used in \cite{KMRW2016}*{Theorem 2.3} we first prove the following 
 result.
  
 \begin{proposition}
 \label{Prop:isom-CstEq2}
 The \(\Cst\)\nb-algebras \(\Contvin(\Specn)\rtimes_{\alpha}\Z\) and~\(\Contvin(\Specn)\rtimes_{\beta}\Z\) are isomorphic.
\end{proposition}
\begin{proof}
Let \(q=\Mod{q}e^{i\theta}\) be the polar decomposition of \(q\). For any 
element~\(\lambda\in\Specn\), define
\[
g_{\theta}(\lambda)=
\begin{cases}
\lambda e^{-i\theta \log_{\Mod{q}} \Mod{\lambda}} & \lambda\neq 0,\\
0 & \lambda=0.
\end{cases}
\]
If~\(\lambda=\Mod{q}^{m}e^{i\psi}\neq 0\), then~\(g_{\theta}(\lambda)=\lambda e^{-im\theta}\). 
Thus \(g_{\theta}\) is continuous at all nonzero 
\(\lambda\in\Specn\). Let \(\{\lambda_k\}\) be a nonzero sequence in \(\Specn\) such that 
\(\lambda_k\to 0\) as \(k\to\infty\). Then
\[
g_{\theta}(\lambda_k)=\lambda_k e^{-i\theta \log_{\Mod{q}} |\lambda_k|} 
\qquad\text{for all~\(\Mod{\lambda_{k}}\neq 0\).}
\]
Since \(\Mod{e^{-i\theta \log_{\Mod{q}} \Mod{\lambda_k}}}\) is bounded for all~\(\Mod{\lambda_k}\neq 0\) we have
\(g_{\theta}(\lambda_{k})\to 0=g_{\theta}(0)\) as~\(k\to\infty\). Hence, \(g_{\theta}\) is continuous at all \(\lambda\in\Specn\). 
Also, \(g_{\theta}\) separates points of \(\Specn\) and \(\lim_{\lambda\to\infty} \Mod{g(\lambda)}=+\infty\). Therefore, 
\(g_{\theta}\) generates \(\Contvin(\Specn)\), see \cite{W1995}*{Section 3, Example 2}. 

Similarly, \(g_{-\theta}\) is also continuous and it is the inverse of~\(g_{\theta}\).
So, \(g_{\theta}\) is a homeomorphism of~\(\Specn\). Therefore, \(g_{\theta}(n)\) also generates~\(\Contvin(\Specn)\).
A simple computation gives
\[
  g_{\theta}(q\lambda)
=q\lambda e^{-i\theta \log_{\Mod{q}}\Mod{q\lambda}}
=qe^{-i\theta}\lambda e^{-i\theta \log_{\Mod{q}}\Mod{q\lambda}}
=\Mod{q}g_{\theta}(\lambda).
\]
Then using functional calculus and the commutation relation \eqref{eq:gen-E2} we obtain
\[
vg(n)v^*=g(qn)=\Mod{q}g_{\theta}(n).
\]
This gives a canonical~\(\Z\) action, denoted by~\(\beta\), on~\(\Contvin(\Specn)\) 
replacing \(q\) by~\(\Mod{q}\) in~\eqref{eq:Z-action} and the 
map~\(n\mapsto g_{\theta}(n)\) is~\(\Z\)\nb-equivariant; hence it extends to an 
isomorphism of crossed products~\(\Contvin(\Specn)\rtimes_{\alpha}\Z\) 
and~\(\Contvin(\Specn)\rtimes_{\beta}\Z\).
\end{proof} 

\subsection{The underlying \(\textup{C}^{*}\)-algebra}
For any closed densely defined operator~\(T\) acting on a Hilbert space 
\(\Hils[K]\) the \(z\)\nb-transform~\(z_{T}\in\Bound(\Hils[K])\) of~\(T\) is defined 
by~\(z_{T}\defeq T(1_{\Hils[K]}+T^*T)^{-\frac{1}{2}}\). Moreover, \(T\) is affiliated with a 
\(\Cst\)\nb-algebra~\(E\), denoted by~\(T\Aff E\), if~\(z_{T}\in\Mult(E)\) and 
\((1_{\Hils[K]}-z_{T}^{*}z_{T})^{\frac{1}{2}}E\) is dense in~\(E\). If~\(T\in\Bound(\Hils[K])\) and \(T\Aff E\), then 
\(T\in\Mult(E)\), see~\cite{W1991b}*{Example 1}.

Consider the crossed product~\(\Cst\)\nb-algebra~\(B=\Contvin(\Specn)\rtimes_{\alpha}\Z\).
Recall that~\(B\) is generated by~\(v\) and~\(n\) in the sense of~\cite{W1995}*{Definition 3.1}. 
In particular, this means~\(v\in\Mult(B)\) and \(n\Aff B\). Now
\(z_{vn}=vz_{n}\) implies~\(vn\Aff B\). Therefore, the operators 
\(v, vn\Aff B\). It is also easy to verify that the pair of operators~\((v,vn)\) 
satisfy the conditions of~\cite{W1995}*{Theorem 3.3}; hence~\(v,vn\) also 
generate~\(B\). 

Since \(vn\Aff B\) and~\(\QuantExp\in\Mult(\Contvin(\Specn))\) 
we get \(n^{-1}vP\otimes vn=X \Aff \Comp(\Hils[L])\otimes B\) 
and consequently~\(\QuantExp(X)\in\U(\Comp(\Hils[L])\otimes B)\). Also observe that 
\[
 \Corep{Y}=(1_{\Hils\otimes\Hils}\otimes v^{*})^{(\hat{N}\otimes 1_{\Hils}\dotplus 1_{\Hils}\otimes\hat{N})\otimes 1_{\Hils[L]}}
 \in\U(\Hils\otimes\Hils\otimes\Hils[L])\cong\U(\Hils[L]\otimes\Hils[L]),
\] 
where~\(\hat{N}\) is the self adjoint densely defined operator acting on 
\(\Hils\) defined in Section~\ref{sec:bcstqgpT}. It was already observed in the proof of 
Proposition~\ref{prop:BrMultunit} that~\(\Corep{Y}\) is a multiplicative unitary. 
It is also easy to verify that~\(\Corep{Y}\) is manageable 
with~\(Q_{\Hils[L]}=\Id_{\Hils[L]}\) and~\(\widetilde{\Corep{Y}}(\overline{e_{i,j}}\otimes e_{k,l}) 
\defeq \overline{e_{i,j}}\otimes e_{k+i+j,l}\), and also generates~\(\T\). 
Thus~\(\Corep{Y}\in\U(\Comp(\Hils[L])\otimes \Cont(\T))\subset 
\U(\Comp(\Hils[L])\otimes B)\), and consequently~\(\BrMultunit\in\U(\Comp(\Hils[L])\otimes B)\). 

On the other hand, 
\(B'=\{(\omega\otimes\Id_{\Hils[L]})\BrMultunit\mid \omega\in \Bound(\Hils[L])_{*}\}^{\textup{CLS}}\) 
is a~\(\Cst\)\nb-algebra. Therefore~\(B'\subseteq\Mult(B)\) and 
\begin{align*}
 B'B &=\{(\omega\otimes\Id_{\Comp(\Hils[L])})\BrMultunit (1_{\Hils[L]}\otimes b)\mid \omega\in\Bound(\Hils[L])_{*}, b\in B\}^\CLS \\
       &=\{(\omega\otimes\Id_{\Comp(\Hils[L])})\BrMultunit (m\otimes b)\mid m\in\Comp(\Hils[L]), \omega\in\Bound(\Hils[L])_{*}, b\in B\}^\CLS \\ 
       &=\{(\omega\otimes\Id_{\Comp(\Hils[L])})(m\otimes b)\mid m\in\Comp(\Hils[L]), \omega\in\Bound(\Hils[L])_{*}, b\in B\}^\CLS 
         = B.
\end{align*}
 
\begin{proposition}
 The \(\Cst\)-algebra \(B'\) coincides with \(B\).
\end{proposition}

\begin{proof}
It is sufficient to show that the operators \(v\) and \(vn\) are affiliated with \(B'\). Because this will 
imply the faithful representation~\(\varphi\colon B\hookrightarrow\Bound(\Hils[L])\) in~\eqref{eq:con-gen-E2} 
is an element of~\(\Mor(B,B')\) and hence \(BB'=B'\).

We consider the following 
family of unitaries \(\{T'(\lambda)\}_{\lambda\in\Specn}\) on~\(\U(\Hils[L]\otimes\Hils[L]\otimes\Hils[L])\) defined by 
 \begin{align}
  \label{eqn:T-prime-Gamma}
  T'(\lambda)&=\QuantExp(\lambda n^{-1}vP\otimes P\otimes vn)\Corep{Y}_{13}.
 \end{align}
 Using~\eqref{eq:braiding} and~\eqref{eq:con-gen-E2} we compute 
  \begin{align}
   \label{eq:j_1vn}
      \Braiding{}{}(vn\otimes 1_{\Hils[L]})\Braiding{}{}^{*} e_{i,j}\otimes e_{k,l} 
      = \zeta^{lj} \Braiding{}{}(vn\otimes 1_{\Hils[L]}) e_{k,l}\otimes e_{i,j}
    &=\zeta^{lj} q^{k} \Braiding{}{}e_{k-1,l+1}\otimes e_{i,j} \nonumber \\
    &=(P\otimes vn) e_{i,j}\otimes e_{k,l}.
  \end{align}
 Combining this with Proposition~\ref{prop:BrMultunit} we get
 \[
    (\BrMultunit^{\lambda})^{*}_{12}\BrMultunit_{23}\BrMultunit^{\lambda}_{12}\BrMultunit_{23}^{*}
   =\Braiding{}{}_{23}\BrMultunit^{\lambda}_{12}\Braiding{}{}^{*}_{23}
   =\QuantExp(\lambda n^{-1}vP\otimes P\otimes vn)\Corep{Y}_{13}
   =T'(\lambda).
 \] 
Expression at the extreme left of the above chain of equalities belongs to \(\U(\Comp(\Hils[L])\otimes\Comp(\Hils[L])\otimes B')\), 
so is \(\QuantExp(\lambda n^{-1}vP\otimes P\otimes vn)\Corep{Y}_{13}\) for all \(\lambda\in\Specn\). Now 
\(\Corep{Y}\in\U(\Hils[L]\otimes\Hils[L])\) and \(X\Aff \Comp(\Hils[L])\otimes\Comp(\Hils[L])\) with~\(\textup{Spec}(X)=\Specn\). 
Then, by virtue of~\cite{W1992a}*{Proposition 5.2}, the map~\(\Specn\ni\lambda\to\QuantExp(\lambda X)\in \Mult(\Comp(\Hils[L]) 
\otimes\Comp(\Hils[L]))\) is strictly continuous and so is the map~\(\lambda\to\BrMultunit^{\lambda}=\BrMultunit(\lambda X)\Corep{Y}\). 
Therefore, \(\{T'(\lambda)\}_{\lambda\in\Specn}\) is a strictly continuous family of elements of \(\U(\Comp(\Hils[L])\otimes\Comp(\Hils[L])\otimes B')\). 
In particular, for \(\lambda=0\) we have \(T'(0)=\Corep{Y}_{13}\); hence \(\Corep{Y}\in\U(\Comp(\Hils[L])\otimes B')\). Now 
the slices \((\omega\otimes\Id_{\Comp(\Hils[L])})\Corep{Y}\) for~\(\omega\in\Bound(\Hils[L])_{*}\) are dense in 
\(\Cont(\T)\) and~\(v\in\Cont(\T)\) imply~\(v\in\Mult(B')\). Also the map 
\(\Specn\ni\lambda\to \QuantExp(\lambda n^{-1}vP\otimes P\otimes vn)
=T'(\lambda)T'(0)^{*}\in\U(\Comp(\Hils[L])\otimes\Comp(\Hils[L])\otimes B')\) is strictly continuous. 
Finally, combining \cite{W1992a}*{Proposition 5.2} and \cite{R2016}*{Proposition 6.11} give  
\((n^{-1}vP\otimes P\otimes vn) \Aff \Comp(\Hils[L])\otimes\Comp(\Hils[L])\otimes B'\) 
and consequently~\(vn\Aff B'\). 
\end{proof}

\subsection{Construction of the comultiplication map}
 Recall the faithful representation~\(\varphi\colon B\to\Bound(\Hils)\) in~\eqref{eq:con-gen-E2}, 
 the unitary operators~\(z\in\U(\Hils)\) and~\(\Corep{U}\in\U(\Hils[L]\otimes\Hils)\) from the 
 Section~\ref{sec:bcstqgpT}. Also recall the actions of~\(z\in\U(\Hils)\) and~\(\Corep{U}\in\U(\Hils[L]\otimes\Hils)\) 
 on the canonical basis elements~\(e_{p}\) and~\(e_{i,j}\otimes e_{p}\)  are defined 
 by~\(z(e_{p})=e_{p+1}\) and~\(\Corep{U}(e_{i,j}\otimes e_{p})=e_{i,j}\otimes e_{p+j}\), 
 respectively. We observe that~\(\Corep{U}\) 
 commutes with~\(v\otimes 1_{\Hils}\) and 
 \[
    \Corep{U}(n\otimes 1_{\Hils})\Corep{U}^{*}(e_{i,j}\otimes e_{p})
  =q^{i} e_{i,j+1}\otimes e_{p+1}
  =n\otimes z (e_{i,j}\otimes e_{p}).
 \]
 Therefore the map~\(\delta\colon B\to B\otimes \Cont(\T)\) given by 
\begin{equation}
 \label{eq:T-act_Eq2}
  \delta(v)\defeq v\otimes 1_{\Cont(\T)} ,
  \qquad
  \delta(n)=n\otimes z.
\end{equation}
is a well defined action of~\(\T\) on~\(B\). Consequently,~\((B,\delta)\) is an object of~\(\Cstcat(\T)\) and \((\Corep{U},\varphi)\) 
 is a faithful covariant representation of~\((B,\delta,\T)\) 
 on~\(\Hils[L]\). Next we define the 
 canonical embeddings~\(j_{1}, j_{2}\in\Mor(B,B\boxtimes_{\Rmat}B)\) 
 in~\eqref{eq:def_j1_j2_Rmat} using the braiding unitary~\(\Braiding{}{}\) 
 defined in~\eqref{eq:braiding}. On the generators~\(v\) and~\(n\)
 
 \begin{equation}
 \label{eq:j_on_E2}
 \begin{aligned}
  j_{1}(v) &\defeq v\otimes 1_{\Hils[L]},
  &\quad
   j_{2}(v) &\defeq \Braiding{}{}(v\otimes 1_{\Hils[L]})\Braiding{}{}^{*}=Z(1_{\Hils[L]}\otimes v)Z^{*}=1_{\Hils[L]}\otimes v,\\
  j_{1}(n) &\defeq n\otimes 1_{\Hils[L]},
  &\quad
   j_{2}(n) &\defeq \Braiding{}{}(n\otimes 1_{\Hils[L]})\Braiding{}{}^{*}=Z(1_{\Hils[L]}\otimes n)Z^{*}=P\otimes n, 
 \end{aligned}
 \end{equation}
 where~\(P\in\U(\Hils[L])\) is defined in~\eqref{eq:P-def}.
 Observe that~\(j_{1}(n), j_{2}(n)\Aff B\boxtimes_{\Rmat}B\).
 
 Define~\(\Comult[B](b)\defeq \BrMultunit (b\otimes 1_{\Hils[L]})\BrMultunit ^{*}\) for all~\(b\in B\). 
 By Theorem~\ref{the:BQgrp-quasitriag}, \(\Bialg{B}\) is the braided 
 \(\Cst\)\nb-quantum group over~\(\T\) generated by~\(\BrMultunit\). 

 Let us compute~\(\Comult[B]\) on the generators~\(v,n\).  Since \(\Corep{Y}\) generates~\(\T\) then \(\Corep{Y}(v\otimes 1_{\Hils[L]})\Corep{Y}^{*}
 =v\otimes v\). Also \(v\otimes v(e_{i,j}\otimes e_{k,l})=e_{i-1,j}\otimes e_{k-1,l}\) shows that 
 it commutes with both the operators~\(\Mod{X}\) and~\(\Ph{X}\) in~\eqref{eq:def_X}, so with 
 \(X\). Therefore,  
 \[
      \Comult[B](v)\defeq \BrMultunit (v\otimes 1_{\Hils[L]})\BrMultunit^{*}
   = \QuantExp(X)\Corep{Y}(v\otimes 1_{\Hils[L]})\Corep{Y}^{*}\QuantExp(X)^{*}
   =v\otimes v
   =j_{1}(v)j_{2}(v).
 \] 

 A simple computation shows that~\(R\defeq j_{1}(n)j_{2}(v^{*})=n\otimes v^{*}\), 
 \(S\defeq j_{1}(v)j_{2}(n)=vP\otimes n\), 
 and 
 \begin{equation}
  \label{eq:commRS}
   X=R^{-1}S=(j_1(n)j_2(v^*))^{-1}j_1(v)j_2(n)=n^{-1}vP\otimes vn. 
 \end{equation}
 Then~\(R\) and~\(S\) are unbounded densely defined  normal operators 
 and satisfy the commutation relations~\cite{W1992a}*{(0.1)}. Since~\(X=R^{-1}S\) is 
 also a normal normal operator with spectrum~\(\Specn\), \cite{W1992a}*{Theorem 2.1-2.2} apply and show that 
 \(R\dotplus S\) is normal with spectrum~\(\Specn\) and 
 \[
     \QuantExp(X) (n\otimes v^{*}) \QuantExp(X)^{*}
   =\QuantExp(R^{-1}S)R\QuantExp(R^{-1}S)^{*}
   =R\dotplus S
   =n\otimes v^{*}\dotplus vP\otimes n.
 \] 
The following computation 
 \begin{equation}
  \label{eq:V-n-comm}
      \Corep{Y}(n\otimes 1_{\Hils[L]})\Corep{Y}^{*} (e_{i,j}\otimes e_{k,l}) 
   = q^{i} e_{i,j+1}\otimes e_{k+1,l}
   = n\otimes v^{*} (e_{i,j}\otimes e_{k,l}).
 \end{equation}
 gives~\(\Comult[B](n)\defeq \BrMultunit (n\otimes 1_{\Hils[L]}) \BrMultunit^{*}= j_{1}(n)j_{2}(v^{*})\dotplus j_{1}(v)j_{2}(n)\). Now, 
 \(n\Aff B\) and~\(\Comult[B]\in\Mor(B,B\boxtimes_{\Rmat}B)\) imply~\(\Comult[B](n)\Aff B\boxtimes_{\Rmat}B\). Thus 
 \(\Comult[B]\) in~\eqref{eq:comult-bE2} extends to a coassociative 
 element of~\(\Mor^{\T}(B, B\boxtimes_{\Rmat}B)\) and satisfies the cancellation conditions. 
 In summary, we obtain the following result.
 \begin{theorem}
  \label{the:BrdqE2}
  The pair~\(\Bialg{B}\) is the braided \(\Cst\)\nb-quantum group over~\(\T\) 
  generated by the braided multiplicative unitary~\(\BrMultunit\). 
 \end{theorem}
 
\begin{remark}
 \label{rem:closed-qnt}
Define~\(p\in\Mor(B,B)\) by~\(p(v)=v\) and~\(p(n)=0\). Existence of~\(p\) is guaranteed by the universal property of~\(B\). Clearly, \(\textup{Im}(p)=\Cont(\T)\). Furthermore, the restriction of the action~\(\delta\) of~\(\T\) on~\(\Cont(\T)\) is trivial and \eqref{eq:j_on_E2} implies \(\Comult[B](v)=v\otimes v=\Comult[\Cont(\T)](v)\). This shows that \(p\in\Mor^{\T}(B,\Cont(\T))\) and satisfies~\((p\boxtimes_{\Rmat} p)\Comult[B](v)=\Comult[\Cont(\T)](p(v))\). Hence, \(\T\) is a \emph{closed quantum subgroup} of braided~\(\qE\).
 \end{remark}

 If \(q\in\R\), then~\(\zeta=1\) which implies 
 the braiding operator~\(\Braiding{}{}=\Flip\) in~\eqref{eq:braiding}.
 Consequently, \(\BrMultunit\) is an ordinary manageable multiplicative unitary, 
 \(\boxtimes_{\Rmat}\) coincides with~\(\otimes\) and ~\eqref{eq:comult-bE2} also coincides 
 with~\eqref{eq:comult-E2}. In conclusion, we have
 \begin{corollary}
  \label{eq:Eq2Wor}
   For real deformation parameters~\(0<q<1\),  the deformation 
   \(\Bialg{B}\) coincide with Woronowicz's~\(\qE\) groups.
 \end{corollary}
 
 \subsection{The bosonisation}  
  The \(\Rmattxt\)\nb-matrix in~\eqref{eq:bichar} corresponds to the group homomorphism~\(\Z\ni m\to\Rmat(\cdot,m)\in\widehat{\Z}\cong\T\) and it induces unique representation~\(\Corep{V}\in\U(\Comp(\Hils[L])\otimes\Contvin(\Z))\subset\U(\Hils[L]\otimes\Hils)\) of~\(\Z\) on~\(\Hils[L]\) defined by~\(\Corep{V}(e_{i,j}\otimes e_{p})=\zeta^{-pj} e_{i,j}\otimes e_{p}\) satisfying~\eqref{eq:du-corep-R-mat}. Define 
  \(\DuCorep{V}\in\U(\Contvin(\Z)\otimes\Comp(\Hils[L]))\subset\U(\Hils\otimes\Hils[L])\) by 
  \begin{equation}
   \label{eq:ducorep-T}
   \DuCorep{V} e_{p}\otimes e_{i,j}= \zeta^{pj}e_{p}\otimes e_{i,j}.
  \end{equation}
  By virtue of Theorem~\ref{the:BQgrp-quasitriag} the managable braided multiplicative unitary \(\BrMultunit\in\U(\Hils[L]\otimes\Hils[L])\) constructed in Theorem~\ref{theorem:main} over~\(\T\) relative to~\((\corep{U},\Rmat)\) is also a manageable braided multiplicative unitary over~\(\T\) in the sense of~\cite{MRW2017}*{Definition 3.2}.

  According to~\cite{MRW2017}*{Theorem 3.7 \& 3.8}   
  \begin{equation}
   \label{eq:MU-boson}
   \mathcal{W}\defeq \Multunit_{13}\Corep{U}_{23}\DuCorep{V}^{*}_{34}\BrMultunit_{24}\DuCorep{V}_{34} 
   \qquad\text{in~\(\U(\Hils\otimes\Hils[L]\otimes\Hils\otimes\Hils[L])\)}
  \end{equation}
  is a manageable multiplicative unitary, where~\(\Corep{U}\) is the 
  concrete realisation of~\(\corep{U}\) on~\(\Hils[L]\otimes\Hils\). 
  Let~\(\mathcal{W}\) generates the \(\Cst\)\nb-quantum group~\(\Qgrp{H}{C}\).  
  
  The \(\Cst\)\nb-algebra~\(C\defeq \Cont(\T)\boxtimes_{\Rmat}B=j_{\Cont(\T)}(\Cont(\T))j_{B}(B)\subset\Bound(\Hils\otimes\Hils[L])\) where, \(j_{\Cont(\T)}(z)=z\otimes 1_{\Hils[L]}\) and 
  \(j_{B}(b)=\DuCorep{V}^{*}(1_{\Hils}\otimes b)\DuCorep{V}\), where~\(z\) is the unitary generator of~\(\Cont(\T)\) and~\(b\in B\subset\Bound(\Hils[L])\). 
The concrete realisations of~\(v, n, z\) give
  \begin{align*}
    j_{B}(v) &=1_{\Hils}\otimes v, 
    &\quad
    j_{\Cont(\T)}(z)j_{B}(v) &=j_{B}(v) j_{\Cont(\T)}(z),\\
   j_{B}(n) &= P' \otimes n,
   &\quad
    j_{\Cont(\T)}(z)j_{B}(n) &=\zeta j_{B}(n) j_{\Cont(\T)}(z),
  \end{align*}
  where~\(P'\in\U(\Hils)\) defined by~\(P' e_{p}=\zeta^{-p}e_{p}\).
  Equivalently, \(C=B\rtimes\Z\) with respect to the action of \(\Z\) on~\(B\) 
  given by \((k,v)\to v\) and~\((k,n)\to \zeta^{k}n\) for all~\(k\in\Z\).
  
   Recall the identification~\(\Hils[L]\cong\Hils\otimes\Hils\). We have already observed 
   \(\Corep{U}=\Multunit_{23}\in\U(\Hils^{\otimes 3})\). Using these we 
   rewrite~\eqref{eq:MU-boson} as
   \begin{equation}
    \label{eq:MU-boson-1}
     \mathcal{W}=\Multunit_{14}\Multunit_{34}\DuCorep{V}^{*}_{456}\QuantExp(n^{-1}vP\otimes vn)_{2356}\Multunit_{25}\Multunit_{35}\DuCorep{V}_{456}
   \quad\text{in~\(\U(\Hils^{\otimes 6})\).} 
  \end{equation}
  Since~\(\DuCorep{V}_{456}\) acts trivially on the fifth leg it commutes with~\(\Multunit_{25}\Multunit_{35}\). Moreover, 
  \[
    \DuCorep{V}^{*}(1_{\Hils}\otimes vn) \DuCorep{V}= 
    j_{B}(vn)=P'\otimes vn.
  \]
  Thus, \eqref{eq:MU-boson-1} gets simplified as 
  \[
   \mathcal{W}=\Multunit_{14}\Multunit_{34}\QuantExp(n^{-1}vP\otimes P'\otimes vn)_{23456}\Multunit_{25}\Multunit_{35}
   \quad\text{in~\(\U(\Hils^{\otimes 6})\).} 
  \]
  and the comultiplication map~\(\Comult[C]\in\Mor(C,C\otimes C)\) is given by 
  \(\Comult[C](c)=\mathcal{W}(c\otimes 1_{\Hils^{\otimes 3}})\mathcal{W}^{*}\) for all~\(c\in C\). Clearly, 
  \[
   \Comult[C](j_{\Cont(\T)}(z)) = j_{\Cont(\T)}(z)\otimes j_{\Cont(\T)}(z), 
   \quad \Comult[C](j_{B}(v)) = j_{B}(v)\otimes j_{B}(v).
 \]
 Next we compute~\(\Comult[C](j_{B}(n))\). By \eqref{eq:V-n-comm} we have
 \[
   \Multunit_{25}\Multunit_{35}(j_{B}(n)\otimes 1_{\Hils^{\otimes 3}})
   \Multunit^{*}_{35}\Multunit^{*}_{25}
   =P'\otimes n\otimes 1_{\Hils}\otimes v^{*},
 \]
  Furthermore, a variant of~\eqref{eq:commRS} gives 
  \begin{align*}
   & \QuantExp(n^{-1}vP\otimes P'\otimes vn)_{23456}(P'\otimes n\otimes 1_{\Hils}\otimes v^{*})\QuantExp(n^{-1}vP\otimes P'\otimes vn)_{23456}^{*}\\
   &= P'\otimes n\otimes 1_{\Hils}\otimes v^{*}\dotplus P'\otimes vP\otimes P' \otimes n.
  \end{align*}
  Now using the concrete realisation of the operators~\(v,n,P,P',\Multunit\) we compute
  \begin{align*}
   \Multunit_{14}\Multunit_{34}(P'\otimes n\otimes 1_{\Hils}) 
   e_{p}\otimes e_{i,j}\otimes e_{s} 
   &=\zeta^{-p} q^{i} e_{p}\otimes e_{i,j+1}\otimes e_{s+j+1+p}\\
   &=((P'\otimes n\otimes z)\Multunit_{14}\Multunit_{34})e_{p}\otimes e_{i,j}\otimes e_{s},
  \end{align*}
  and
  \begin{align*}
  \Multunit_{14}\Multunit_{34}(P'\otimes vP\otimes P') e_{p}\otimes e_{i,j}\otimes e_{q} 
  &=\zeta^{-p-s-j}e_{p}\otimes e_{i-1,j}\otimes e_{s+p+j}\\
  &= ((1_{\Hils}\otimes v\otimes P')\Multunit_{14}\Multunit_{34})e_{p}\otimes e_{i,j}\otimes e_{s}.
  \end{align*}
 Combining the last four calculations we obtain
 \[
   \Comult[C](j_{B}(n))
   =j_{B}(n)\otimes j_{\Cont(\T)}(z)j_{B}(v^{*}) \dotplus  j_{B}(v)\otimes j_{B}(n).
 \]
Summarising, we have the bosonisation of~\(\qE\).
\begin{theorem}
 \label{the:boson}
 Let \(C\) be the universal \(\Cst\)\nb-algebra generated by the the unitaries~\(\tilde{z},\tilde{v}\) and the normal operator~\(\tilde{n}\) with~\(\Sp(\tilde{n})=\Specn\) subject to the commutation relations 
 \[
   \tilde{z}\tilde{v}\tilde{z}^{*}=\tilde{v},
   \quad 
   \tilde{z}\tilde{n}\tilde{z}^{*}=\zeta\tilde{n},
   \quad
   \tilde{v}\tilde{n}\tilde{v}^{*}=q\tilde{n}.
 \]
 There exists a unique~\(\Comult[C]\in\Mor(C,C\otimes C)\) such that
 \[
  \Comult[C](\tilde{z})=\tilde{z}\otimes\tilde{z}, 
  \quad 
  \Comult[C](\tilde{v})=\tilde{v}\otimes\tilde{v},
  \quad 
  \Comult[C](\tilde{n})=\tilde{n}\otimes\tilde{z}\tilde{v}^{*}\dotplus \tilde{v}\otimes\tilde{n}
 \] 
 and~\(\Qgrp{H}{C}\) is a \(\Cst\)\nb-quantum group. Moreover, there exists an idempotent Hopf~\Star{}homomorphism~\(f\in\Mor(C,C)\) with~\(f(\tilde{z})=\tilde{z}\), 
 \(f(\tilde{v})=1_{C}\) and~\(f(\tilde{n})=0\). Its image is the copy of \(\Cont(\T)\) generated 
 by~\(\tilde{z}\) as a closed quantum subgroup of~\(\G[H]\) and its kernel is the copy of~\(B\) generated by~\(\tilde{v},\tilde{n}\) as the braided~\(\qE\) group over~\(\T\).
\end{theorem}

 \section{Contraction procedure between braided \(\textup{SU}_{q}(2)\) and \(\textup{E}_{q}(2)\) groups and their respective bosonisations}
  \label{sec:contraction}
  Throughout this section we fix a complex deformation parameter~\(0<|q|<1\), 
  the unitary \(\Rmattxt\)\nb-matrix~\(\Rmat\colon\Z\times\Z\to\T\), defined 
  by~\eqref{eq:bichar} and \(\T\) is a quasitriangular compact quantum group with respect to~\(\Rmat\). We denote the braided~\(\qE\) constructed 
  in Theorem~\ref{the:BrdqE2} by \((B_{\qE},\Comult[\qE])\).

Denote~\(\N^{0}=\N\cup\{0\}\). Consider the Hilbert space \(\Hils[L]_{\qSU}=\ell^{2}(\N^{0}\times\Z)\) equipped with an orthonormal basis~\(\{e_{i,j}\}_{i\in\N^{0}; j\in\Z}\). Denote~\(\Hils[L]_{\qE}=\ell^{2}(\Z\times\Z)\) and fix an orthonormal basis~\(\{e_{i,j}\}_{i,j\in\Z}\) for it. Recall that~\(\Hils[L]_{\qE}\) is an object of~\(\Corepcat(\T)\) with respect to the representation \(\Corep{U}(e_{i,j}\otimes e_{p})=e_{i,j}\otimes e_{p+j}\) defined in Section~\ref{sec:bcstqgpT}. Identification of the basis vectors of~\(\Hils[L]_{\qSU}\) with the corresponding basis vectors of~\(\Hils[L]_{\qE}\) defines an embedding~\(\Hils[L]_{\qSU}\hookrightarrow\Hils[L]_{\qE}\). Furthermore, the restriction of~\(\Corep{U}\) on 
\(\Hils[L]_{\qSU}\) defines a representation of~\(\T\) on it. Consequently, we obtain a \(\T\)\nb-equivariant embedding \(\Bound(\Hils[L]_{\qSU})\hookrightarrow\Bound(\Hils[L]_{\qE})\) and~\(1_{\Hils[L]_{\qSU}}\) is a \(\T\)\nb-equivariant orthogonal projection onto~\(\Hils[L]_{\qSU}\).   
  
  Now we recall the braided \(\qSU=(B_{\qSU},\Comult[\qSU])\) 
  group over \(\T\) with respect to~\(\Rmat\), constructed in~\cite{KMRW2016}. Define~\(\alpha,\gamma\in\Bound(\Hils[L]_{\qSU})\) by 
  \begin{equation}
     \alpha e_{i,j}\defeq \sqrt{1-|q|^{2i}}e_{i-1,j}, 
     \qquad
     \gamma e_{i,j}\defeq q^{i} e_{i,j-1}.
  \end{equation}
  Then \(B_{\qSU}\) is the universal \(\Cst\)\nb-algebra generated by~\(\alpha\) and 
  \(\gamma\). Define \(f_{\alpha}, f_{\gamma}\in\Contvin(\Specn)\) by 
\(f_{\alpha}(\lambda)=\sqrt{1-\Mod{\lambda}^{2}}\chi(\lambda)\) 
and~\(f_{\gamma}(\lambda)=\overline{\lambda}\chi(\lambda)\), 
where \(\chi\) is the indicator function of the 
closed unit disc~\(\{z\in\C\mid \Mod{z}\leq 1\}\). By the property of the 
continuous functional calculus we have
\begin{equation}
 \label{eq:SU2-E2-gen}
 \alpha=vf_{\alpha}(n), 
 \qquad 
 \gamma=f_{\gamma}(n).
\end{equation}
Following similar arguments as in~\cite{W1992}*{Section 1} and replacing \(\mu\) by \(\Mod{q}\) it is easy to observe
\begin{equation}
 \label{eq:contr-alg}
 B_{\qSU}\subset B_{\qE} 
 \quad\text{and}\quad
 B_{\qSU}=1_{\Hils[L]_{\qSU}}B_{\qE}1_{\Hils[L]_{\qSU}}.
\end{equation}
Then the restriction of~\(\delta\) in~\eqref{eq:T-act_Eq2} 
\(\delta\colon B_{\qSU}\to B_{\qSU}\otimes 
  \Cont(\T)\) defines an action of~\(\G\) on~\(B_{\qSU}\). It is easy to verify 
  that~\(\delta(\alpha)=\alpha\otimes 1_{\Cont(\T)}\) and 
  \(\delta(\gamma)=\gamma\otimes z^{*}\) is an action of~\(\T\) on~\(B_{\qSU}\), where 
  \(z\) is the unitary generator of~\(\Cont(\T)\) defined in the beginning of Section~\ref{sec:bcstqgpT}. 
  In fact, \(\delta(b)=\Corep{U}(b\otimes 1_{\Cont(\T)})\Corep{U}^{*}\) for 
  all~\(b\in B_{\qSU}\). 
  
 Consequently \(B_{\qSU}\subset B_{\qE}\) is \(\T\)\nb-equivariant, 
\( B_{\qSU}\boxtimes_{\Rmat} B_{\qSU}\subset  B_{\qE}\boxtimes_{\Rmat} B_{\qE}\) and the embeddings of~\(B_{\qSU}\) into~\(B_{\qSU}\boxtimes_{\Rmat} B_{\qSU}\) are obtained by restricting~\(j_{1}, j_{2}\in\Mor(B_{\qE}, B_{\qE}\boxtimes_{\Rmat} B_{\qE})\) in~\eqref{eq:j_on_E2}. In fact, a simple computation 
using~\eqref{eq:SU2-E2-gen} and~\eqref{eq:P-def} give \(j_{1}(\alpha)=\alpha\otimes 1_{\Hils[L]_{\qE}}\), \(j_{2}(\alpha)=1_{\Hils[L]_{\qE}}\otimes \alpha\), \(j_{1}(\gamma)=\gamma\otimes 1_{\Hils[L]_{\qE}}\) and 
\(j_{2}(\gamma)=P^{*}\otimes\gamma\), where~\(P\in\U(\Hils[L]_{\qE})\) defined 
in~\eqref{eq:P-def}.

The comultiplication map~\(\Comult[\qSU]\colon B_{\qSU}\to B_{\qSU}\boxtimes_{\Rmat} B_{\qSU}\) is defined by 
\begin{equation}
 \label{eq:Comult_qSU}
  \begin{aligned}
 \Comult[\qSU](\alpha) &= j_{1}(\alpha) j_{2}(\alpha) - qj_{1}(\gamma)^{*} j_{2}(\gamma),\\
 \Comult[\qSU](\gamma) &= j_{1}(\gamma) j_{2}(\alpha) + j_{1}(\alpha)^{*} j_{2}(\gamma).
\end{aligned}
\end{equation}
Consider the family of inner automorphisms~\(\{\tau^{k}\}_{k\in\Z}\) of~\(B_{\qE}\) define by~\(\tau^{k}(a)=v^{k}a v^{-k}\) for all~\(k\in\Z\). By 
virtue of~\eqref{eq:gen-E2} we have
 \[
   \tau^{k}(v)=v, 
   \qquad 
   \tau^{k}(n)=q^{k}n.
 \]
Therefore, \(\tau^{k}\) is \(\T\)\nb-equivariant, that is, \((\tau^{k}\otimes\Id_{\Cont(\T)})\circ\delta=\delta\circ\tau^{k}\) for all~\(k\in\Z\). Similarly, 
using~\eqref{eq:comult-bE2} we observe that \(\tau^{k}\) is an 
automorphism of braided \(\Cst\)\nb-quantum groups: 
\begin{equation}
 \label{eq:bq-hom}
      (\tau^{k}\boxtimes_{\Rmat}\tau^{k})\circ\Comult[\qE] 
   = \Comult[\qE]\circ\tau^{k} 
   \qquad\text{for all~\(k\in\Z\).}
\end{equation}
 The braided analogue of the contraction procedure~\cite{W1992} is contained 
 is following result. 
 \begin{theorem}
  \label{the:contraction}
  For any~\(a\in B_{\qE}\) 
  \begin{equation}
   \label{eq:contraction}
   \Comult[\qE](a) 
   =\lim_{k\to\infty} (\tau^{k}\boxtimes_{\Rmat}\tau^{k})\Comult[\qSU]\bigl(1_{\Hils[L]_{\qSU}}(\tau^{-k}(a))1_{\Hils[L]_{\qSU}}\bigr).
  \end{equation} 
 \end{theorem} 
 \begin{proof}
 The proof essentially follows from \cite{W1992}*{Sections 2 \& 3} replacing \(\mu\) by \(\Mod{q}\), tensor product~\(\otimes\) of \(\Cst\)\nb-algebras by~\(\boxtimes_{\Rmat}\) and taking care of certain commutation relations.
 
 Recall the dense \Star{}subalgebra~\(\mathcal{B}_{\qE}\) of~\(B_{\qE}\) 
 defined in~\cite{W1992}*{Equation 22}
 \[
   \mathcal{B}_{\qE}=\cup_{k\in\Z} \tau^{l}(B_{\qSU}).
\]   
 Suppose \(a\in \mathcal{B}_{\qE}\). For sufficiently large~\(k\), 
 \(\tau^{-k}(a)\in B_{\qSU}\), which implies, \(\Comult[\qSU](\tau^{-k}(a))\in 
 B_{\qSU}\boxtimes_{\Rmat} B_{\qSU}\). Also, 
 \(a\) commutes with~\(1_{\Hils[L]_{\qSU}}\). 
 Therefore, \eqref{eq:contraction} coincides with the following expression
\begin{equation}
 \label{eq:contraction-aux}
 \Comult[\qE](a) 
   =\lim_{k\to\infty} (\tau^{k}\boxtimes_{\Rmat}\tau^{k})\Comult[\qSU](\tau^{-k}(a))
   \quad
   \text{for all~\(a\in \mathcal{B}_{\qE}\).}
\end{equation}
Define the following elements
\begin{equation}
 \label{eq:def-Y-t}
\begin{aligned}
  t &=\prod_{k=1}^{\infty}\bigl(1_{\Hils[L]_{\qSU}}-\Mod{q}^{2k}\gamma^{*}\gamma\bigr)\in B_{\qSU},\\
  Y &=\sum_{r=0}^{\infty} \Bigl(\prod_{i=1}^{r} \frac{1}{1-\Mod{q}^{2i}}\Bigr) 
  \bigl(-q j_1(\gamma^{*}) j_{2}(\gamma)\bigr)^{r}\bigl(j_{1}(v)j_{2}(v)\bigr)^{-r}
  \in B_{\qE}\boxtimes_{\Rmat} B_{\qE}.
\end{aligned}
\end{equation}
Note that the first term corresponding to \(r=0\) is~\(1_{\qSU}\boxtimes_{\Rmat}1_{\qSU}\).  
Clearly, \(\delta(\gamma^{*}\gamma)=\gamma^{*}\gamma\otimes 1\). Equivalently, 
\(\gamma^{*}\gamma\) is a degree zero \(\T\)\nb-homogeneous element of~\(B_{\qSU}\). Similarly, each term in~\(Y\) is also degree zero \(\T\)\nb-homogeneous 
element with respect to the diagonal action~\(\delta\bowtie\delta\) of~\(\T\) on 
\(B_{\qE}\boxtimes_{\Rmat} B_{\qE}\) mentioned in the Section~\ref{sec:Quasitriag}.
Therefore, both the elements~\(t\) and~\(Y\) are degree zero \(\T\)\nb-homogeneous elements. We also observe that~\(j_{k}(\alpha)j_{l}(\gamma)=j_{l}(\gamma)j_{k}(\alpha)\) for all~\(k,l =1,2\). 

These give the braided analogue of the formulae~\cite{W1992}*{Equation 42-43}:
\[
  \lim_{k\to\infty}\alpha^{k}v^{-k}= t^{\frac{1}{2}},
  \quad
  \lim_{k\to\infty}\Comult[\qSU](\alpha^{k})(j_{1}(v)j_{2}(v))^{-k}=j_{1}(t^{\frac{1}{2}})j_{2}(t^{\frac{1}{2}})Y.
\]
Consequently, for~\(a\in B_{\qSU}\) we  prove that the right hand side of~\eqref{eq:contraction-aux} is well defined 
\begin{align*}
 & \lim_{k\to\infty} \Bigl[(j_{1}(v)j_{2}(v))^{k}\Comult[\qSU]
 \bigl(v^{-k}t^{\frac{1}{2}}\bigl(t^{-\frac{1}{2}}at^{-\frac{1}{2}}\bigr)t^{\frac{1}{2}}v^{k}\bigr)(j_{1}(v)j_{2}(v))^{-k}\Bigr]\\ 
 &=\lim_{k\to\infty} \Bigl[(j_{1}(v)j_{2}(v))^{k}\Comult[\qSU]
 \bigl((\alpha^{*})^{k}\bigl(t^{-\frac{1}{2}}at^{-\frac{1}{2}}\bigr)\alpha^{k}\bigr)(j_{1}(v)j_{2}(v))^{-k}\Bigr]\\
 &=\widetilde{Y}^{*}\Comult[\qSU](a)\widetilde{Y}, 
 \quad\text{where~\(\widetilde{Y}\defeq\Comult[\qSU](t^{-\frac{1}{2}})(j_{1}(t^{\frac{1}{2}}) j_{2}(t^{\frac{1}{2}}))Y\in B_{\qE}\boxtimes_{\Rmat}B_{\qE}\).}
\end{align*}
Let~\(\Comult'(a)=\widetilde{Y}^{*}\Comult[\qSU](a)\widetilde{Y}\) for all~\(a\in B_{\qSU}\). For any~\(l\in\N\) and~\(a\in\mathcal{B}_{\qE}\), \eqref{eq:contraction-aux} implies 
\((\tau^{l}\boxtimes_{\Rmat}\tau^{l})\Comult'(a)=\Comult' (\tau^{l}(a))\). Also, \(\tau^{l} f(n)\to f(0)1_{\Hils[L]_{\qE})}\) almost uniformly, abbreviated as a.u., as~\(l\to\infty\) 
for all~\(f\in\Contvin(\Specn)\). In particular, we have
\begin{equation}
 \label{eq:act-on-qSU}
 \text{a.u.}\lim_{l\to\infty} \tau^{l}(\alpha)=v, 
 \qquad 
 \text{a.u.}\lim_{l\to\infty} \tau^{l}(\gamma)=0,
\end{equation}
Using the second limit in~\eqref{eq:act-on-qSU} we obtain 
\begin{equation}
 \label{eq:au-lim-Y-t}
  \text{a.u.}\lim_{l\to\infty}(\tau^{l}\boxtimes_{\Rmat}\tau^{l})Y= 1_{\Hils[L]_{\qE}\otimes\Hils[L]_{\qE}}, 
 \qquad 
  \text{a.u.}\lim_{l\to\infty}\tau^{l}(t^{\frac{1}{2}})= 1_{\Hils[L]_{\qE}}.
\end{equation}
Consequently, the bounded sequence~\(\{a_{l}\defeq \tau^{l}(t)\}_{l\in\N}\) in~\(B_{\qE}\) satisfies~\cite{W1992}*{Equation (39)}
\[
  \text{a.u.}\lim_{l\to\infty}\Comult'(\tau^{l}(t))
 = \text{a.u.}\lim_{l\to\infty}(\tau^{l}\boxtimes_{\Rmat}\tau^{l})(Y^{*}(t\otimes t)Y)
 = 1_{\Hils[L]_{\qE}\otimes\Hils[L]_{\qE}}.
\] 
Thus, \(\Comult'\) uniquely extends to an element in \(\Mor^{\T}(B_{\qE},B_{\qE}\boxtimes_{\Rmat}B_{\qE})\).

To complete the proof of~\eqref{eq:contraction-aux}, it is sufficient to verify 
\begin{equation}
 \label{eq:cont-aux-2}
  \lim_{l\to\infty}\Comult'(\tau^{l}(t^{\frac{1}{2}}\alpha t^{\frac{1}{2}}))=\Comult[\qE](v), 
 \qquad 
  \lim_{l\to\infty}\Comult'(\tau^{l}(t^{\frac{1}{2}}\gamma^{*} t^{\frac{1}{2}}))
  =\Comult[\qE](n).    
\end{equation}
From the concrete realisation of~\(n\) and~\(P\) given by~\eqref{eq:con-gen-E2} and 
\eqref{eq:P-def}, it is easy to verify that~\(nP\) is a normal operator and~\(\textup{Sp}(nP)=\Specn\). Furthermore, \eqref{eq:j_on_E2} immediately implies \(j_{1}(n)j_{2}(n)=nP\otimes n\). As immediate consequence of~\cite{W1992}*{Proposition 1.1} we obtain 
\begin{equation*}
  \lim_{l\to\infty} \Mod{q}^{-l}\norm{(\tau^{l}(\gamma^{*})\boxtimes_{\Rmat}\tau^{l}(\gamma))\psi}=0, 
 \quad\text{for all~\(\psi\in\dom(j_{1}(n))\cap\dom(j_{2}(n))\).}
\end{equation*}
This implies 
\begin{equation}
 \label{eq:dom-Y-aux}
 \lim_{l\to\infty} \Mod{q}^{-l}\norm{(\tau^{l}\boxtimes_{\Rmat}\tau^{l})Y\psi-\psi}=0, 
 \quad\text{for all~\(\psi\in\dom(j_{1}(n))\cap\dom(j_{2}(n))\).}
\end{equation}
Then the proof of~\eqref{eq:cont-aux-2} follows similarly from~\cite{W1992}*{Section 3 (Step 3)} using \eqref{eq:Comult_qSU}, \eqref{eq:def-Y-t}--\eqref{eq:dom-Y-aux}. 
\end{proof}

We end this section with an application of the contraction procedure described 
above to show that the bosonisations of~\(\qSU\) and~\(\qE\) are also related via a contraction procedure for the respective braided quantum groups. 

Recall the bosonisation~\(\Qgrp{H}{C}\) of~\(\qE\) described in Theorem~\ref{the:boson}. For any \(l\in\Z\) set \(\tilde{\tau}^l\defeq 1_{\Cont(\T)} \boxtimes_{\Rmat}\tau^l\). Then \(\{\tilde{\tau}^l\}_{l\in\Z}\) is also a one parameter group of automorphisms of \(C=\Cont(\T)\boxtimes_{\Rmat}B_{\qE}\). Furthermore, \(\tilde{\tau}^{-k}(c)\in B_{\qU}\) for \(c\in C\) and sufficiently large \(k\).

Bosonisation of~\(\qSU\) is the compact quantum group~\(\qU=(C_{\qU},\Comult[\qU])\) 
with \(\T\) as a closed quantum subgroup~\cite{KMRW2016}*{Section 6}. The underlying~\(\Cst\)\nb-algebra~\(C_{\qU}\) is the twisted tensor product \(\Cont(\T)\boxtimes_{\Rmat}B_{\qSU}\subset\Bound(\Hils\otimes\Hils[L]_{\qSU})\subset\Bound(\Hils\otimes\Hils[L]_{\qE})\). The canonical embeddings~\(j_{\Cont(\T)}, j_{B_{\qSU}}\) are give by \(j_{\Cont(\T)}(z)=z\otimes 1_{\Hils[L]_{\qE}}\) and \(j_{B_{\qSU}}(a)=\DuCorep{V}^{*}(1_{\Hils}\otimes a)\DuCorep{V}\) for \(a\in B_{\qSU}\subset\Bound(\Hils[L]_{\qSU})\), where~\(\DuCorep{V}\) is the unitary in~\eqref{eq:ducorep-T}. Consequently, \eqref{eq:contr-alg} implies
\[
  C_{\qU}\subset C 
  \quad\text{and}\quad
  C_{\qU}=1_{\Hils\otimes\Hils[L]_{\qSU}} C 1_{\Hils\otimes\Hils[L]_{\qSU}} .
\]
From~\cite{MRW2016}*{Theorem 6.4}, the 
comultiplication~\(\Comult[\qU]\) is given by \(\Comult[\qU]\defeq \Psi\circ (\Id_{\Cont(\T)}\boxtimes_{\Rmat}\Comult[\qSU])\), where 
\[
  \Psi(x)\defeq \Multunit_{13}\Corep{U}_{23}\DuCorep{V}_{34}^{*}x_{124}\DuCorep{V}_{34}\Corep{U}_{23}^{*}\Multunit[*]_{13}
  \quad \text{ for } x\in\Bound(\Hils\otimes\Hils[L]_{\qE}\otimes\Hils[L]_{\qE}).
\]

We also observe that \(\Comult[C]=\Psi\circ (\Id_{\Cont(\T)} \boxtimes_{\Rmat} \Comult[\qE])\). 
In particular, we get~\(\Comult[C]|_{\Cont(\T)}=\Comult[\qU]|_{\Cont(\T)}=\Comult[\Cont(\T)]\) as expected.
\begin{corollary}
 Contraction of~\(\qU\) group is isomorphic to the bosonisation of 
 braided~\(\qE\) group. Equivalently, 
\begin{equation}
\label{eq:contr-boson}
\Comult[C](c)= \lim_{k\to \infty}(\tilde{\tau}^{k}\otimes\tilde{\tau}^{k})
\Comult[\qU]\Bigl(1_{\Hils[H]\otimes\Hils[L]_{\qE}}\bigl(\tilde{\tau}^{-k}(c)\bigr)1_{\Hils[H]\otimes\Hils[L]_{\qE}}\Bigr)
\quad\text{for all~\(c\in C\).}
\end{equation}
\end{corollary}
\begin{proof}
Since \(\mathcal{B}_{\qE}\) is dense in \(B_{\qE}\), 
\(\mathcal{C}\defeq j_{\Cont(\T)}(\Cont(\T)j_{B_{\qE}}(\mathcal{B}_{\qE})\) is also 
dense in \(C\) and the equation \eqref{eq:contr-boson} coincides with the following:
\begin{equation}
 \label{eq:cont-U2}
\Comult[C](c)= \lim_{k\to \infty}(\tilde{\tau}^{k}\otimes\tilde{\tau}^{k})\Comult[\qU](\tilde{\tau}^{-k}(c)) 
\qquad \text{for \(c\in\mathcal{C}\).}
\end{equation}
For every~\(k\in\Z\), the automorphisms 
\(\tilde{\tau}^{k}\) act trivially on the first factor of~\(C=\Cont(\T)\boxtimes_{\Rmat} B_{\qE}\). 
Subsequently, \eqref{eq:cont-U2} becomes the equality of the restrictions 
of~\(\Comult[C]\) and~\(\Comult[\qU]\) on the common 
closed quantum subgroup~\(\T\) for~\(c=j_{\Cont(\T)}(z)\). 

Restriction of the faithful representation~\(j_{B_{\qE}}\in\Mor(B_{\qE},\Comp(\Hils\otimes\Hils[L]_{\qE}))\) on~\(\Cont(\T)\) is given by 
\(j_{B_{\qE}}(v)=1_{\Hils}\otimes v\). Then 
\(\tilde{\tau}^{k}(c)=(1_{\Hils}\otimes v^{k})c(1_{\Hils}\otimes v^{-k})\) for 
\(c\in C\subset\Bound(\Hils\otimes\Hils[L]_{\qE})\) and~\(k\in\Z\). It is also easy to verify that~\(v\otimes 1_{\Hils}\) commutes with~\(\Corep{U}\),~\(1_{\Hils}\otimes v\) commutes with~\(\DuCorep{V}\). These imply
\[
  (\tilde{\tau}^{k}\otimes\tilde{\tau}^{k})\Psi(x)
  =\Psi\bigl((\Id_{\Cont(\T)}\boxtimes_{\Rmat}\tau^{k}\boxtimes_{\Rmat}\tau^{k})(x)\bigr)
 \quad\text{for all~\(x\in \Cont(\T)\boxtimes_{\Rmat} B_{\qE}\boxtimes_{\Rmat} B_{\qE}\).}
 \]
Finally, using~\eqref{eq:contraction-aux} we verify~\eqref{eq:cont-U2} for 
\(c= j_{B_{\qE}}(b)\) with \(b\in\mathcal{B}_{\qE}\)
 \begin{align*}
      \Comult[C](j_{B_{\qE}}(b)) 
 &=\Psi (1_{\Cont(\T)} \boxtimes_{\Rmat} \Comult[\qE](b))\\
 &=\Psi \Bigl(1_{\Cont(\T)} \boxtimes_{\Rmat} \lim_{k\to\infty} \bigl((\tau^{k}\boxtimes_{\Rmat}\tau^{k})\Comult[\qSU](\tau^{-k}(b))\bigr)\Bigr)\\
 &=\lim_{k\to\infty} \Psi \Bigl(1_{\Cont(\T)}\boxtimes_{\Rmat} \bigl((\tau^{k}\boxtimes_{\Rmat}\tau^{k})\Comult[\qSU](\tau^{-k}(b))\bigr)\Bigr)\\
 &=\lim_{k\to\infty} \Bigl((\tilde{\tau}^{k}\otimes\tilde{\tau}^{k})\Psi\bigl(1_{\Cont(\T)} 
 \boxtimes_{\Rmat}\Comult[\qSU](\tau^{-k}(b))\bigr)\Bigr)\\
 &=\lim_{k\to\infty} \Bigl((\tilde{\tau}^{k}\otimes\tilde{\tau}^{k})
 \Comult[\qU]\bigl(j_{B_{\qSU}}(\tau^{-k}(b))\bigr)\Bigr)\\
 &=\lim_{k\to\infty} \Bigl((\tilde{\tau}^{k}\otimes\tilde{\tau}^{k})
  \Comult[\qU]\bigl(\tilde{\tau}^{-k}(j_{B_{\qSU}}(b))\bigr)\Bigr). 
  \qedhere
 \end{align*}
\end{proof}

\appendix
\section{Yetter-Drinfeld representation category over quasitriangular quantum groups}
\label{sec:YD-rep}  
 Let \(\Qgrp{G}{A}\) be a quantum group, \(\DuQgrp{G}{A}\) be its dual, and  
 \(\multunit\in\U(\hat{A}\otimes A)\) be the reduced bicharacter. 
 
 A~\(\G\)\nb-\emph{Yetter-Drinfeld representation} is a triple 
 \((\Hils[L], \corep{U},\corep{V})\) consisting of a Hilbert space 
 \(\Hils[L]\), representations~\(\corep{U}\) and~\(\corep{V}\) of 
 \(\G\) and \(\DuG\) on~\(\Hils[L]\) subject to the commutation relation:
    \begin{equation}
      \label{eq:U_V_compatible_1}
      \corep{V}_{12}\corep{U}_{13} \multunit_{23}
      = \multunit_{23} \corep{U}_{13}\corep{V}_{12}
      \quad\text{in }\U(\Comp(\Hils[L])\otimes\hat{A}\otimes A).
    \end{equation}
 A morphism between~\(\G\)\nb-Yetter-Drinfeld representations 
 \((\Hils[L]_{1}, \corep{U}^{1},\corep{V}^{1})\) 
 and~\((\Hils[L]_{2}, \corep{U}^{2},\corep{V}^{2})\) is an element 
 \(t\in\Bound(\Hils[L]_{1},\Hils[L]_{2})\) such that 
 \(t\in\Hom^{\G}(\corep{U}^{1},\corep{U}^{2})\) in~\(\Corepcat(\G)\) 
 and~\(t\in\Hom^{\DuG}(\corep{V}^{1},\corep{V}^{2})\) in~\(\Corepcat(\DuG)\), 
 respectively. Let  \(\YDCorepcat(\G)\) denotes the category of \(\G\)\nb-Yetter-Drinfeld 
 representations. It is easy to verify that~\(\tenscorep\) operation 
 on \(\YDCorepcat(\G)\) defined by \((\Hils[L]_{1}, \corep{U}^{1},\corep{V}^{1})
 \tenscorep (\Hils[L]_{2}, \corep{U}^{2},\corep{V}^{2})\defeq 
 (\Hils[L]_{1}\otimes\Hils[L]_{2}, \corep{U}^{1}\tenscorep\corep{U}^{2},\corep{V}^{1}
 \tenscorep\corep{V}^{2})\) is a tensor product with tensor unit~\(\C\). This category 
 is already a unitarily braided monoidal category, see~\cite{MRW2016}*{Proposition 3.2}.

\begin{proposition}
 \label{prop:YD-vs-Corep-R}
 Let~\(\Qgrp{G}{A}\) be a quasitriangular quantum group with an~\(\Rmattxt\)\nb-matrix 
 \(\Rmat\in\U(\hat{A}\otimes\hat{A})\). For any object~\((\Hils[L],\corep{U})\) in~\(\Corepcat(\G)\) 
 the~\(\Rmattxt\)\nb-matrix~\(\Rmat\) induces a canonical object~\((\Hils[L],\corep{V})\) in 
 \(\Corepcat(\DuG)\) such that~\((\Hils[L],\corep{U},\corep{V})\) becomes an object in 
 \(\YDCorepcat(\G)\). Moreover, the construction gives an injective braided 
 monoidal functor~\(\mathcal{F}\colon\Corepcat(\G)\to\YDCorepcat(\G)\) that maps 
 \((\Hils[L],\corep{U})\) to~\((\Hils[L],\corep{U},\corep{V})\) and leaves the morphisms 
 unchanged. 
\end{proposition}
\begin{proof}
Since \(\Rmat\in\U(\hat{A}\otimes\hat{A})\) is a bicharacter, it is also a quantum group 
homomorphism from~\(\G\) to~\(\DuG\). Let~\(\Delta_{R}\in 
\Mor(A,A\otimes\hat{A})\) be the right quantum group homomorphism associated to it
~\cite{MRW2012}*{Theorem 5.3}. Let~\((\Hils[L],\corep{U})\) be an object in \(\Corepcat(\G)\).  
Consider the unitary~\(\widetilde{V}\defeq \corep{U}^{*}_{12} 
\big((\Id_{\Comp(\Hils[L])}\otimes\Delta_{R})\corep{U}\bigr) 
\in\U(\Comp(\Hils[L])\otimes A\otimes\hat{A})\). In particular, \(\Delta_{R}\) satisfies 
\((\Comult[A]\otimes\Id_{\hat{A}})\circ\Delta_{R}=(\Id_{A}\otimes\Delta_{R})\circ\Comult[A]\). 
Using the last equation and the fact the~\(\corep{U}\) is a representation, we compute
\begin{align*}
       (\Id_{\Comp(\Hils[L])}\otimes\Comult[A]\otimes\Id_{\hat{A}})\widetilde{V}
   &= \bigl((\Id_{\Comp(\Hils[L])}\otimes\Comult[A])\corep{U}^{*}\bigr)_{123}
        \Bigl(\bigl(\Id_{\Comp(\Hils[L])}\otimes\bigl((\Comult[A]\otimes\Id_{\hat{A}})\Delta_{R}\bigr)\bigr)\corep{U}\Bigr)\\
  &=\corep{U}^{*}_{13}\corep{U}^{*}_{12} \Bigl(\bigl(\Id_{\Comp(\Hils[L])}\otimes\bigl((\Id_{A}\otimes\Delta_{R})\Comult[A]\bigr)\bigr)\corep{U}\Bigr)\\
  &=\corep{U}^{*}_{13}\corep{U}^{*}_{12}\bigl((\Id_{\Comp(\Hils[L])}\otimes\Id_{A}\otimes\Delta_{R})\corep{U}_{12}\corep{U}_{13}\bigr)\\
  &=\corep{U}^{*}_{13}((\Id_{\Comp(\Hils[L])}\otimes\Delta_{R})\corep{U})_{134}
     =\widetilde{V}_{134}.   
\end{align*}
By \cite{MRW2012}*{Corollary 2.2} the second leg of~\(\widetilde{V}\) is trivial; hence there is a 
unique element~\(\corep{V}\in\U(\Comp(\Hils[L])\otimes\hat{A})\) such 
that 
\begin{equation}
 \label{eq:du-corep-R-mat}
(\Id_{\Comp(\Hils[L])}\otimes\Delta_{R})\corep{U}=\corep{U}_{12}\corep{V}_{13}  
\qquad\text{in~\(\U(\Comp(\Hils[L])\otimes A\otimes\hat{A})\).}
\end{equation}
Also, the first part of the proof of~\cite{MRW2012}*{Theorem 5.3} shows that~\((\Hils[L],\corep{V})\) is 
an object in~\(\Corepcat(\DuG)\). 

Now we show that~\((\Hils[L],\corep{U},\corep{V})\) is an object  in~ \(\YDCorepcat(\G)\). The second 
condition in~\eqref{eq:redbichar} and the~\(\Rmattxt\)\nb-matrix condition~\eqref{eq:R-mat} together imply
\begin{align*}
     \multunit_{23}\multunit_{13}\Rmat_{12} 
 =\bigl((\DuComult[A]\otimes\Id_{A})\multunit\bigr)\Rmat_{12} 
 =\Rmat_{12}\bigl((\flip\circ\DuComult[A]\otimes\Id_{A})\multunit\bigr)
 &=\Rmat_{12}\bigl(\flip_{12}( \multunit_{23}\multunit_{13})\bigr)\\
 &=\Rmat_{12}\multunit_{13}\multunit_{23}.
\end{align*}
The bijective correspondence between~\(\Rmat\) and \(\Delta_{R}\) in~\cite{MRW2012}*{Equation 32}, 
the first condition in~\eqref{eq:redbichar} and the last identity give
\begin{align*}
     \multunit_{34}\Bigl(\flip_{34}\bigl(\Id_{\hat{A}}\otimes(\Comult[A]\otimes\Id_{A})\Delta_{R})\multunit\bigr)\Bigr)\multunit[*]_{34}
 &=\multunit_{34}\Bigl(\flip_{34}\bigl(\multunit_{12}\multunit_{13}\Rmat_{14}\bigr)\Bigr)\multunit[*]_{34}\\
 &=\multunit_{34}\multunit_{12}\multunit_{14}\Rmat_{13}\multunit[*]_{34}\\
 &=\multunit_{12}\multunit_{34}\multunit_{14}\Rmat_{13}\multunit[*]_{34}\\
 &=\multunit_{12}\Rmat_{13}\multunit_{14}\\
 &=(\Id_{\hat{A}}\otimes(\Delta_{R}\otimes\Id_{A})\Comult[A])\multunit .
\end{align*}
Then slicing the first leg of the last computation by~\(\omega\in\hat{A}'\) we obtain
\[
  \multunit_{23}\Bigl(\flip_{23}\bigl((\Comult[A]\otimes\Id_{A})\Delta_{R}(a)\bigr)\Bigr)\multunit[*]_{23}=(\Delta_{R}\otimes\Id_{A})\Comult[A](a)
  \quad\text{for all~\(a\in A\).}
\]
Using this we now compute
\begin{align*}
           \corep{U}_{12}\corep{V}_{13}\corep{U}_{14}\multunit_{34}
     &=\bigl((\Id_{\Comp(\Hils[L])}\otimes(\Delta_{R}\otimes\Id_{A})\Comult[A])\corep{U}\bigr)\multunit_{34}\\
     &=\multunit_{34}\Bigl(\flip_{34}\bigl((\Id_{\Comp(\Hils[L])}\otimes(\Comult[A]\otimes\Id_{A})\Delta_{R})\corep{U}\bigr)\Bigr)\\
     &=\multunit_{34}\Bigl(\flip_{34}\bigl(\corep{U}_{12}\corep{U}_{13}\corep{V}_{14}\bigr)\Bigr)
     =\multunit_{34}\corep{U}_{12}\corep{U}_{14}\corep{V}_{13}
     =\corep{U}_{12}\multunit_{34}\corep{U}_{14}\corep{V}_{13}.
\end{align*}
Cancelling the unitary~\(\corep{U}_{12}\) on both sides of the last equation we obtain~\eqref{eq:U_V_compatible_1}; 
hence~\((\Hils[L],\corep{U},\corep{V})\) is an object in~\(\YDCorepcat(\G)\).

Suppose~\((\Hils[L]_{i},\corep{U}^{i})\) and~\((\Hils[L]_{i},\corep{V}^{i})\) are objects in~\(\Corepcat(\G)\) and~\(\Corepcat(\DuG)\) respectively, induced by the~\(\Rmattxt\)\nb-matrix of~\(\G\) for~\(i=1,2\). Let 
\(t\in\Hom^{\G}(\corep{U}^{1},\corep{U}^{2})\) in~\(\Corepcat(\G)\). Then
\begin{align*}
     \corep{U}^{2}_{12}(t\otimes 1_{A\otimes\hat{A}})\corep{V}^{1}_{13}
   =  (t\otimes 1_{A\otimes\hat{A}})\corep{U}^{1}_{12}\corep{V}^{1}_{13}
  &=(t\otimes\Delta_{R})\corep{U}\\
  &=((\Id_{\Comp(\Hils[L]_{2})}\otimes\Delta_{R})\corep{U}^{2})(t\otimes 1_{A\otimes\hat{A}})\\
  &=\corep{U}^{2}_{12}\corep{V}^{2}_{13}(t\otimes 1_{A\otimes\hat{A}}).
\end{align*}
Cancelling the unitary~\(\corep{U}^{2}_{12}\) on both sides of the last equation we get 
\(t\in\Hom^{\DuG}(\corep{V}^{1},\corep{V}^{2})\) in~\(\Corepcat(\DuG)\). Thus we 
have an injective functor~\(\mathcal{F}\colon \Corepcat(\G)\to\YDCorepcat(\G)\)  
that maps objects \((\Hils[L],\corep{U})\to (\Hils[L],\corep{U},\corep{V})\) and 
leaves the morphisms unchanged. Furthermore, \(\mathcal{F}\) preserves the tensor 
product of representations:
\begin{align*}
 (\Id_{\Comp(\Hils[L]_{1}\otimes\Hils[L]_{2})}\otimes\Delta_{R})\corep{U}^{1}\tenscorep\corep{U}^{2} 
 &= (\Id_{\Comp(\Hils[L]_{1}\otimes\Hils[L]_{2})}\otimes\Delta_{R})(\corep{U}^{1}_{13}\corep{U}^{2}_{23})\\
 &=\corep{U}^{1}_{13}\corep{V}^{1}_{14}\corep{U}^{2}_{23}\corep{V}^{2}_{24}\\
 &=\corep{U}^{1}_{13}\corep{U}^{2}_{23}\corep{V}^{1}_{14}\corep{V}^{2}_{24}
    =(\corep{U}^{1}\tenscorep\corep{U}^{2})_{123}(\corep{V}^{1}\tenscorep\corep{V}^{2})_{124}.
\end{align*}
In particular, the proof of~\cite{MRW2016}*{Theorem 5.3} shows that~\(\mathcal{F}\) preserves the braidings. 
Therefore, \(\mathcal{F}\) is an injective braided monoidal structure preserving functor.  
\end{proof}


\begin{bibdiv}
  \begin{biblist}
\bib{B1992}{article}{
  author={Baaj, Saad},
  title={Repr\'{e}sentation r\'{e}guli{\`e}re du groupe quantique {$E_\mu (2)$} de {W}oronowicz},
  date={1992},
  issn={0764-4442},
  journal={C. R. Acad. Sci. Paris S\'{e}r. I Math.},
  volume={314},
  number={13},
  pages={1021\ndash 1026},
  eprint={http://gallica.bnf.fr/ark:/12148/bpt6k58688425/f1025.item},
  review={\MR {1168528}},
}

\bib{BS1993}{article}{
  author={Baaj, Saad},
  author={Skandalis, Georges},
  title={Unitaires multiplicatifs et dualit\'{e} pour les produits crois\'{e}s de {$C^*$}-alg{\`e}bres},
  date={1993},
  issn={0012-9593},
  journal={Ann. Sci. \'{E}cole Norm. Sup. (4)},
  volume={26},
  number={4},
  pages={425\ndash 488},
  eprint={http://www.numdam.org/item?id=ASENS_1993_4_26_4_425_0},
  review={\MR {1235438}},
}

\bib{CGST1990}{article}{
   author={Celeghini, Enrico},
   author={Giachetti, Riccardo},
   author={Sorace, Emanuele},
   author={Tarlini, Marco},
   title={Three-dimensional quantum groups from contractions of ${\rm
   SU}(2)_q$},
   date={1990},   
   issn={0022-2488},   
   journal={J. Math. Phys.},
   volume={31},
   number={11},
   pages={2548--2551},
   doi={10.1063/1.529000},
   review={\MR{1075731}},   
}

\bib{CGST1991}{article}{
   author={Celeghini, Enrico},
   author={Giachetti, Riccardo},
   author={Sorace, Emanuele},
   author={Tarlini, Marco},
   title={The quantum Heisenberg group $H(1)_q$},   
   date={1991},
   issn={0022-2488},
   journal={J. Math. Phys.},
   volume={32},
   number={5},
   pages={1155--1158},
   doi={10.1063/1.529311},   
   review={\MR{1103466}},
}

\bib{CGST1991a}{article}{
   author={Celeghini, Enrico},
   author={Giachetti, Riccardo},
   author={Sorace, Emanuele},
   author={Tarlini, Marco},
   title={The three-dimensional Euclidean quantum group $E(3)_q$ and its
   $R$-matrix},
   date={1991},   
   issn={0022-2488},   
   journal={J. Math. Phys.},
   volume={32},
   number={5},
   pages={1159--1165},
   doi={10.1063/1.529312},
   review={\MR{1103467}},   
}

\bib{DKSS2012}{article}{
  author={Daws, Matthew},
  author={Kasprzak, Pawe\l },
  author={Skalski, Adam},
  author={So\l{}tan, Piotr~M.},
  title={Closed quantum subgroups of locally compact quantum groups},
  date={2012},
  issn={0001-8708},
  journal={Adv. Math.},
  volume={231},
  number={6},
  pages={3473\ndash 3501},
  doi={10.1016/j.aim.2012.09.002},
  review={\MR {2980506}},
}

\bib{IW1953}{article}{
   author={In\"on\"u, Erdal},
   author={Wigner, Eugene P.},
   title={On the contraction of groups and their representations},
   journal={Proc. Nat. Acad. Sci. U.S.A.},
   volume={39},
   date={1953},
   pages={510--524},
   issn={0027-8424},
   review={\MR{55352}},
   doi={10.1073/pnas.39.6.510},
}

\bib{KMRW2016}{article}{
  author={Kasprzak, Pawe\l },
  author={Meyer, Ralf},
  author={Roy, Sutanu},
  author={Woronowicz, Stanis\l aw~L.},
  title={Braided quantum {$\rm SU(2)$} groups},
  date={2016},
  issn={1661-6952},
  journal={J. Noncommut. Geom.},
  volume={10},
  number={4},
  pages={1611\ndash 1625},
  doi={10.4171/JNCG/268},
  review={\MR {3597153}},
}

\bib{M1994}{article}{
  author={Majid, Shahn},
  title={Cross products by braided groups and bosonization},
  date={1994},
  issn={0021-8693},
  journal={J. Algebra},
  volume={163},
  number={1},
  pages={165\ndash 190},
  doi={10.1006/jabr.1994.1011},
  review={\MR {1257312}},
}

\bib{M1999}{article}{
   author={Majid, S.},
   title={Double-bosonization of braided groups and the construction of $U_q(\mathfrak{g})$},
   date={1999},   
   issn={0305-0041},   
   journal={Math. Proc. Cambridge Philos. Soc.},
   volume={125},
   number={1},
   pages={151--192},
   doi={10.1017/S0305004198002576},
   review={\MR{1645545}},   
}

\bib{M2000}{article}{
   author={Majid, Shahn},
   title={Braided-Lie bialgebras},
   date={2000},   
   issn={0030-8730},
   journal={Pacific J. Math.},
   volume={192},
   number={2},
   pages={329--356},
   review={\MR{1744574}},
   doi={10.2140/pjm.2000.192.329},
}

\bib{MRW2012}{article}{
  author={Meyer, Ralf},
  author={Roy, Sutanu},
  author={Woronowicz, Stanis\l aw~L.},
  title={Homomorphisms of quantum groups},
  date={2012},
  issn={1867-5778},
  journal={M\"{u}nster J. Math.},
  volume={5},
  pages={1\ndash 24},
  eprint={http://nbn-resolving.de/urn:nbn:de:hbz:6-88399662599},
  review={\MR {3047623}},
}

\bib{MRW2014}{article}{
  author={Meyer, Ralf},
  author={Roy, Sutanu},
  author={Woronowicz, Stanis\l aw~L.},
  title={Quantum group-twisted tensor products of {C{$^*$}}-algebras},
  date={2014},
  issn={0129-167X},
  journal={Internat. J. Math.},
  volume={25},
  number={2},
  pages={1450019, 37},
  doi={10.1142/S0129167X14500190},
  review={\MR {3189775}},
}

\bib{MRW2016}{article}{
  author={Meyer, Ralf},
  author={Roy, Sutanu},
  author={Woronowicz, Stanis\l aw~L.},
  title={Quantum group-twisted tensor products of {${\rm C}^*$}-algebras. {II}},
  date={2016},
  issn={1661-6952},
  journal={J. Noncommut. Geom.},
  volume={10},
  number={3},
  pages={859\ndash 888},
  doi={10.4171/JNCG/250},
  review={\MR {3554838}},
}

\bib{MRW2017}{article}{
  author={Meyer, Ralf},
  author={Roy, Sutanu},
  author={Woronowicz, Stanis\l aw~Lech},
  title={Semidirect products of {$\rm C^*$}-quantum groups: multiplicative unitaries approach},
  date={2017},
  issn={0010-3616},
  journal={Comm. Math. Phys.},
  volume={351},
  number={1},
  pages={249\ndash 282},
  doi={10.1007/s00220-016-2727-3},
  review={\MR {3613504}},
}

\bib{NV2010}{article}{
      author={Nest, Ryszard},
      author={Voigt, Christian},
       title={Equivariant {P}oincar\'{e} duality for quantum group actions},
        date={2010},
        ISSN={0022-1236},
     journal={J. Funct. Anal.},
      volume={258},
      number={5},
       pages={1466\ndash 1503},
         doi={10.1016/j.jfa.2009.10.015},
      review={\MR{2566309}},
}

\bib{R1985}{article}{
  author={Radford, David~E.},
  title={The structure of {H}opf algebras with a projection},
  date={1985},
  issn={0021-8693},
  journal={J. Algebra},
  volume={92},
  number={2},
  pages={322\ndash 347},
  doi={10.1016/0021-8693(85)90124-3},
  review={\MR {778452}},
}

\bib{R2013}{thesis}{
  author={Roy, Sutanu},
  title={{$\rm C^*$}-quantum groups with projection},
  type={phdthesis},
  institution={Georg-August Universit\"at G\"ottingen},
  date={2013},
  eprint={http://hdl.handle.net/11858/00-1735-0000-0022-5EF9-0},
}

\bib{R2016}{article}{
  author={Roy, Sutanu},
  title={Braided quantum groups and their bosonization in the C*-algebraic framework},
  date={2016},
  eprint={https://arxiv.org/abs/1601.00169v4},
}

\bib{S2020}{article}{
  author={So\l{}tan, Piotr~M.},
  title={Podle\'{s} spheres for the braided quantum {SU}(2)},
  date={2020},
  issn={0024-3795},
  journal={Linear Algebra Appl.},
  volume={591},
  pages={169\ndash 204},
  doi={10.1016/j.laa.2020.01.011},
  review={\MR {4053085}},
}

\bib{SW2007}{article}{
  author={So\l{}tan, Piotr~M.},
  author={Woronowicz, Stanis\l aw~L.},
  title={From multiplicative unitaries to quantum groups. {II}},
  date={2007},
  issn={0022-1236},
  journal={J. Funct. Anal.},
  volume={252},
  number={1},
  pages={42\ndash 67},
  doi={10.1016/j.jfa.2007.07.006},
  review={\MR {2357350}},
}

\bib{W1991a}{article}{
  author={Woronowicz, Stanis\l aw~L.},
  title={Quantum {$E(2)$} group and its {P}ontryagin dual},
  date={1991},
  issn={0377-9017},
  journal={Lett. Math. Phys.},
  volume={23},
  number={4},
  pages={251\ndash 263},
  doi={10.1007/BF00398822},
  review={\MR {1152695}},
}

\bib{W1991b}{article}{
  author={Woronowicz, Stanis\l aw~L.},
  title={Unbounded elements affiliated with {$C^*$}-algebras and noncompact quantum groups},
  date={1991},
  issn={0010-3616},
  journal={Comm. Math. Phys.},
  volume={136},
  number={2},
  pages={399\ndash 432},
  review={\MR {1096123}},
}

\bib{W1992a}{article}{
  author={Woronowicz, Stanis\l aw~L.},
  title={Operator equalities related to the quantum {$E(2)$} group},
  date={1992},
  issn={0010-3616},
  journal={Comm. Math. Phys.},
  volume={144},
  number={2},
  pages={417\ndash 428},
  review={\MR {1152380}},
}

\bib{W1992}{article}{
  author={Woronowicz, Stanis\l aw~L.},
  title={Quantum {${\rm SU}(2)$} and {$E(2)$} groups. {C}ontraction procedure},
  date={1992},
  issn={0010-3616},
  journal={Comm. Math. Phys.},
  volume={149},
  number={3},
  pages={637\ndash 652},
  review={\MR {1186047}},
}

\bib{W1995}{article}{
  author={Woronowicz, Stanis\l aw~L.},
  title={{$C^*$}-algebras generated by unbounded elements},
  date={1995},
  issn={0129-055X},
  journal={Rev. Math. Phys.},
  volume={7},
  number={3},
  pages={481\ndash 521},
  doi={10.1142/S0129055X95000207},
  review={\MR {1326143}},
}

\bib{W1996}{article}{
  author={Woronowicz, Stanis\l aw~L.},
  title={From multiplicative unitaries to quantum groups},
  date={1996},
  issn={0129-167X},
  journal={Internat. J. Math.},
  volume={7},
  number={1},
  pages={127\ndash 149},
  doi={10.1142/S0129167X96000086},
  review={\MR {1369908}},
}

\bib{W2001}{article}{
  author={Woronowicz, Stanis\l aw~L.},
  title={Quantum ``{$az+b$}'' group on complex plane},
  date={2001},
  issn={0129-167X},
  journal={Internat. J. Math.},
  volume={12},
  number={4},
  pages={461\ndash 503},
  doi={10.1142/S0129167X01000836},
  review={\MR {1841400}},
}

\bib{WZ2002}{article}{
  author={Woronowicz, Stanis\l aw~L.},
  author={Zakrzewski, S.},
  title={Quantum `{$ax+b$}' group},
  date={2002},
  issn={0129-055X},
  journal={Rev. Math. Phys.},
  volume={14},
  number={7-8},
  pages={797\ndash 828},
%  note={Dedicated to Professor Huzihiro Araki on the occasion of his 70th birthday},
  doi={10.1142/S0129055X02001405},
  review={\MR {1932667}},
}
  \end{biblist}
\end{bibdiv}
\end{document}